\documentclass[a4paper,10pt,notitlepage,reqno]{article}
\usepackage[english]{babel}
\usepackage{amssymb,amsthm,amsmath} 
\usepackage{mathtools}
\usepackage{mathrsfs}
\usepackage{enumerate}
\usepackage{authblk}
\usepackage[noadjust]{cite}
\usepackage[symbol]{footmisc}
\usepackage{soul}
\usepackage[dvipsnames]{xcolor}
\usepackage{geometry}
\usepackage[colorlinks=true, linkcolor=red, citecolor=blue, urlcolor=blue,hyperfootnotes=false]{hyperref}
\usepackage{cleveref}

\theoremstyle{plain}
\newtheorem{theorem}{Theorem}[section]
\newtheorem{lemma}{Lemma}[section]
\newtheorem{proposition}{Proposition}[section]
\newtheorem{corollary}{Corollary}[section]

\theoremstyle{definition}

\theoremstyle{remark}
\newtheorem{remark}{Remark}[section]

\DeclarePairedDelimiter{\abs}{\lvert}{\rvert} 
\DeclarePairedDelimiter{\norm}{\lVert}{\rVert}  

\DeclareMathOperator{\divergenza}{div}
\renewcommand{\div}{\divergenza}

\DeclareMathOperator{\sgn}{sgn}

\DeclareMathOperator{\dist}{dist}

\newcommand{\R}{\mathbb{R}}
\newcommand{\C}{\mathbb{C}}

\newcommand{\N}{\mathbb{N}}
\newcommand{\normeq}[1]{{\left\vert\kern-0.25ex\left\vert\kern-0.25ex\left\vert #1 
    \right\vert\kern-0.25ex\right\vert\kern-0.25ex\right\vert}}

\newenvironment{system}%
{\left\lbrace\begin{array}{r@{\hspace{1mm}}ll}}%
{\end{array}\right.}

\title{\textbf{Eigenvalue bounds and spectral stability of Lamé operators with complex potentials}}

\author[1]{Biagio Cassano}
\author[2]{Lucrezia Cossetti} 
\author[3]{Luca Fanelli}		
	
\affil[1]{Dipartimento di Matematica, Università degli Studi di Bari Aldo Moro, Via Edoardo Orabona 4, 70125 Bari, Italy; biagio.cassano@uniba.it}

\affil[2]{Fakult\"{a}t f\"{u}r Mathematik, Institut f\"{u}r Analysis, Karlsruher Institut f\"{u}r Technologie (KIT), Englerstra{\ss}e 2, 76131 Karlsruhe, Germany; lucrezia.cossetti@kit.edu}

\affil[3]{Ikerbasque \& Departamento de Matem\'aticas, Universidad del Pa\'is Vasco/Euskal Herriko Unibertsitatea (UPV/EHU), Barrio Sarriena s/n, 48940, Leioa, Spain; luca.fanelli@ehu.es}

\begin{document}

\date{\small \today}

\maketitle



\begin{abstract}
  \noindent
	This paper is devoted to providing quantitative bounds on the
        location of eigenvalues, both \emph{discrete} and
        \emph{embedded}, of non self-adjoint Lamé operators of
        elasticity $-\Delta^\ast + V$ in terms of suitable norms of
        the potential $V$. In particular, this allows to get sufficient conditions on the size of the potential such that the point spectrum of the perturbed operator remains empty. In three dimensions we show full spectral stability under suitable form-subordinated perturbations: we prove that the spectrum is purely continuous and coincides with the non negative semi-axis as in the free case.
\end{abstract}

\section{Introduction}

This paper is devoted to the analysis of the spectrum of the perturbed
Lamé operator of elasticity $-\Delta^\ast + V.$
The Lamé operator $-\Delta^\ast$ acts on smooth vector fields as
\begin{equation*}
	-\Delta^\ast u:= 
	-\mu \Delta u									 
	- (\lambda + \mu) \nabla \div u,
	\qquad u\in C^\infty_0(\R^d)^d := C^\infty_0(\R^d;\C^d),
\end{equation*}
where the material-dependent Lamé parameters $\lambda, \mu\in \R$ satisfy the standard ellipticity conditions (\emph{cfr.}~\cite[Sec.~2.2]{Co17})
\begin{equation*}
	\mu>0,\,
	\lambda+ 2\mu>0.
\end{equation*}	
The Lam\'e operator is self-adjoint on $H^1(\R^d)^d$ and 
$\sigma(-\Delta^\ast) = \sigma_{\textup{ac}}(-\Delta^\ast) = [0,
+\infty)$; we refer the reader to~\cite{La_Li, Kupradze} for a detailed
exposition of the general theory of elasticity and to
\cite{B_F-G_P-E_R_V, B_F_R_V_V,B_F-G_P-E_R_V2,B_F-G_P-E_R_V0} and references therein for previous
results in the topic.
We consider the perturbation $V\colon \R^d\to \C^{d\times d}$ to be a multiplication operator by a (possibly)
non-hermitian matrix: this frames our study into a non-self-adjoint
setting.
Spectral analysis of non-self-adjoint models has seen a huge
development in the last decades and nowadays the literature in this
direction is very extensive,
see~\cite{A_A_D,B_T_G20,Ca_Ib_Kr_St19,Cu14,Cuenin,Cu_La_Tr14,
Cu_Si18,DA_Fa_Sc20,Enblom16,En18,
FK19,FKV18,F_K_V2,Frank,FrankIII,F_L_L_S,Fr_Sa17,
FS17,Ha_Kr20,IKL19,Ib_St19,Kr_Ku20,La_Sa,Sa,Sambou16,Davies,
D_H_K,D_H_K13,F_L_S16} which is just a selection of the existing
material in the subject.

The study of the \emph{discrete} spectrum of the non-self-adjoint
Lamé operator $-\Delta^\ast + V$ was started in~\cite{Co19}:
in this paper we extend these results 
to cover \emph{embedded}
eigenvalues. Moreover, we investigate the spectral stability of the
Lam\'e operator of elasticity and get in any dimension $d\geq 1$ sufficient conditions on the
size of the potential  that guarantee that the point spectrum of the perturbed operator remains empty. In the special case $d=3$ we show that the whole spectrum is preserved under suitable form-subordinated perturbations.

Adapting to the Lamé operator new techniques introduced by Frank in~\cite{Frank} for the Laplacian,  in~\cite{Co19} it is shown that every \emph{discrete} eigenvalue $z\in
\C \setminus [0,\infty)$ of $-\Delta^\ast + V$ lies in the closed disk
of the complex plane centered at the origin and with radius whose size depends on the Lebesgue,
Morrey-Campanato or Kerman-Sawyer norm, according to the chosen class of potentials considered. More specifically, when the size of the potential is measured with respect to the $L^p$ topology, \cite[Theorem 1.2]{Co19} shows that any eigenvalue $z\in \C \setminus [0,\infty)$ of $-\Delta^\ast + V$ satisfies
\begin{equation}
  \label{eq:Lp-bound}
  |z|^\gamma \leq
  C \| V \|_{L^{\gamma + \frac{d}{2}}(\R^d)}^{\gamma + \frac{d}{2}},
\end{equation} 
 for some $C>0$, with $d\geq 2$ and $0\leq\gamma\leq 1/2$ ($\gamma\neq 0$ if $d=2$).

In order to cover potentials with stronger local singularities one
considers the Morrey-Campanato class $\mathcal{L}^{\alpha,p}(\R^d)$,
that is the class of functions $W$ such that for $\alpha>0$ and $1\leq p\leq d/\alpha$ the following norm 
\begin{equation*}
	\norm{W}_{\mathcal{L}^{\alpha,p}(\R^d)}:=\sup_{x,r} r^\alpha \Big( r^{-d} \int_{B_r(x)} \abs{W(x)}^p\,dx \Big)^{\frac{1}{p}} 
\end{equation*}
is finite.
For example, $1/|x|^\alpha\notin L^{d/\alpha}(\R^d)=\mathcal{L}^{\alpha, d/\alpha}(\R^d)$
but $1/|x|^\alpha  \in 
\mathcal{L}^{\alpha, p}(\R^d)$ for $\alpha>0$ and $1\leq
p<d/\alpha$. In particular,
the inverse-square potential of quantum mechanics $V(x)=1/|x|^2,$ $x\in \R^3,$ at first ruled out by the $L^p$ type condition, can be recovered once the size of the potential is measured in terms of Morrey-Campanato norms. 
In~\cite[Theorem 1.3]{Co19}, the analogous bound
to~\eqref{eq:Lp-bound} for potentials in the Morrey-Campanato class is
provided: any eigenvalue $z\in \C\setminus [0,\infty)$ of $-\Delta^\ast + V$ satisfies
\begin{equation}\label{eq:MC-bound}
	|z|^\gamma \leq C\|V\|_{\mathcal{L}^{\alpha, p}(\R^d)}^{\gamma + \frac{d}{2}},
\end{equation}
for some $C>0$, $d\geq 2,$ $0\leq\gamma\leq 1/2$ ($\gamma\neq 0$ if $d=2$) and
with $(d-1)(2\gamma + d)/[2(d-2\gamma)]<p\leq \gamma + d/2$ and
$\alpha=2d/(2\gamma +d).$

We remark that for $\alpha >0$ and $1<p\leq d/\alpha$ the condition
$W \in \mathcal{L}^{\alpha,p}(\R^d)$
ensures the $L^2$ weighted boundedness of fractional integrals
(see Fefferman~\cite{Fefferman1983} for the special case $\alpha=2$
and~\cite{Perez95} for the more general result, see
also~\cite[Section~2.2]{B_B_R_V}),
that is the existence of a non-negative constant $C(W)>0$ such that
\begin{equation}\label{eq:est_fract_int}
\|I_{\alpha/2} f\|_{L^2(\R^d , W dx)}\leq C(W)
\|f\|_{L^2(\R^d)}, \quad \text{ for all } f\in C_c^\infty(\R^d),
\end{equation} 
where $\widehat{I_\alpha
  f}(\xi)=|\xi|^{-\alpha}\widehat{f}.$
If $W\in \mathcal{L}^{\alpha, p}(\R^d)$ the constant $C(W)$ in~\eqref{eq:est_fract_int} can be written more explicitly in terms of the Morrey-Campanato norm $\|W\|_{\mathcal{L}^{\alpha,p,d}(\R^d)}$ of $W,$ more specifically
  \begin{equation}\label{eq:MC-fract-int}
  	C(W)=C_{\alpha,p,d} \|W\|_{\mathcal{L}^{\alpha,p,d}(\R^d)}^{1/2},
  \end{equation}
for $C_{\alpha,p,d}>0$ independent on $W.$
The largest class of functions $W$ such that this inequality is available is
the Kerman-Saywer space $\mathcal{KS}_\alpha(\R^d)$
(see~\cite[Theorem 2.3]{Ke_Sa1986}), namely the set of all the functions $W$ such that for $0<\alpha<d$ the following norm
\begin{equation*}
	\norm{W}_{\mathcal{KS}_\alpha(\R^d)}:= \sup_{Q} \Big( \int_{Q} \abs{W(x)}\, dx \Big)^{-1} \int_Q \int_Q \frac{\abs{W(x)} \abs{W(y)}}{\abs{x-y}^{d-\alpha}}\, dx\, dy
\end{equation*}   
is finite (the supremum is taken over all dyadic cubes $Q$ in $\R^d$).
As a matter of fact the finiteness of this norm is a necessary and sufficient condition
for the validity of~\eqref{eq:est_fract_int} and the best constant in
it is
\begin{equation}\label{eq:const.k-s}
C(W)=C_{\alpha,d} \|W
\|_{\mathcal{KS}_\alpha(\R^d)}^{1/2},
\end{equation}
for some
 constant $C_{\alpha,d}>0$ independent on $W$.
In particular this implies $\|W\|_{\mathcal{KS_\alpha}(\R^d)}\leq
C\|W\|_{\mathcal{L}^{\alpha,p}(\R^d)}$,
 for $\alpha>0$, $1<p\leq d/\alpha$ and $C>0$, which gives
 $\mathcal{L}^{\alpha, p}(\R^d)\subseteq \mathcal{KS}_\alpha(\R^d)$. 
 In the case $\alpha=2,$~\eqref{eq:est_fract_int} is equivalent  to the validity of an Hardy-type inequality for the weight $W,$ namely
\begin{equation}\label{eq:Hardy}
	\int_{\R^d} |W||f|^2\, dx\leq a_W \int_{\R^d} |\nabla f|^2\, dx, \quad \text{ for all } f\in C_c^\infty(\R^d),
\end{equation} 
where $a_W:=C(W)^2$ and $C(W)$ is the constant in~\eqref{eq:est_fract_int}.
In the case $d=3$ we have that 
\begin{equation}\label{eq:d=3-a}
	a_W=
	\begin{system}
		c_\textup{F} \|W\|_{\mathcal{L}^{2,p}(\R^3)}, \quad &\text{if}\; W\in \mathcal{L}^{2,p}(\R^3),\\
		c_{\textup{KS}}  \|W\|_{\mathcal{KS}_2(\R^3)}, \quad &\text{if}\; W\in \mathcal{KS}_2(\R^3),
	\end{system}
\end{equation} 
where we have set $c_\textup{F}=c_\textup{F}(p):=C_{2,p,3}^2,$ with $C_{2,p,3}$ as in~\eqref{eq:MC-fract-int} and $c_\textup{KS}:=C_{2,3}^2,$ with $C_{2,3}$ as in~\eqref{eq:const.k-s}.

In \cite[Theorem 1.4]{Co19} it is shown that any eigenvalue $z\in \C \setminus[0,\infty)$ of $-\Delta^\ast + V$ satisfies 
\begin{equation}\label{eq:KS-bound}
	|z|^\gamma \leq  C \, Q_2(|V|)^{2\gamma +d} \| | V |^\beta\|_{\mathcal{KS}_\alpha(\R^d)}^{\frac{1}{\beta}(\gamma + \frac{d}{2})},
\end{equation}
for $C>0$, $d\geq 2,$ $1/3\leq\gamma<1/2$ if $d=2$ and $0\leq \gamma <1/2$ if
$d\geq 3$ and where $\alpha=2d\beta/(2\gamma + d)$ and
$\beta=(d-1)(2\gamma + d)/[2(d-2\gamma)],$ under the additional
assumption that $|V|$ 
belongs to the $A_2(\R^d)$ Muckenhoupt class of weights, \emph{i.e.}, the set of measurable non-negative functions $w$ such that the following quantity
\begin{equation*}
	Q_2(w)
	:=\sup_{Q} \Bigg ( \frac{1}{\abs{Q}} \int_Q w(x)\, dx \Bigg) 
	\Bigg ( \frac{1}{\abs{Q}} \int_Q \frac{1}{w(x)}\, dx\Bigg)
\end{equation*}
is finite. Here the supremum is taken over any cube $Q$ in $\R^d.$

We stress that, in the higher dimensional case $d\geq 3$, the validity
of bounds~\eqref{eq:Lp-bound},~\eqref{eq:MC-bound}
and~\eqref{eq:KS-bound} provides conditions which guarantee the absence of
non-embedded discrete eigenvalues depending on the size of the potential, measured with respect to the corresponding norm. Indeed, once $\gamma=0$ is fixed, 
 for any eigenvalue $z\in \C\setminus [0,\infty)$ of $-\Delta^\ast +V$ one has
\begin{equation}\label{eq:gamma=0}
	1\leq C \|V\|^\frac{d}{2},
\end{equation}	
where $\|V\|$ denotes $\|V\|_{L^{\gamma + \frac{d}{2}}(\R^d)},$ $\|V\|_{\mathcal{L}^{\alpha, p}(\R^d)}$ or $Q_2(|V|)^d\, \||V|^\frac{d-1}{2}\|_{\mathcal{KS}_{d-1}(\R^d)}^\frac{2}{d-1},$ respectively.
 If $ C \|V\|^\frac{d}{2}<1,$ \eqref{eq:gamma=0} yields a contradiction and so $\sigma_\textup{d}(-\Delta^\ast+V)=\varnothing$ (\emph{cfr}.~\cite{Co19}, Thm.~1.2, Cor.~1.1 and Cor.~1.2). 

 Seeking for eigenvalue bounds
like~\eqref{eq:Lp-bound}, \eqref{eq:MC-bound} and \eqref{eq:KS-bound} for perturbed
Lamé operators $-\Delta^\ast + V$ with $V$ possibly
non-hermitian was mainly motivated by the existence in the
literature of the corresponding bounds for non-self-adjoint
Schr\"odinger operators $-\Delta + V$ and by the link between the
two operators given by the Helmoltz
decomposition, see \Cref{lem:helmholtz}.
More motivation come from the one-dimensional framework, where the Lamé operator becomes a constant multiple of the Laplacian, \emph{i.e.}~$\Delta^\ast=(\lambda+2\mu)d^2/dx^2$.
As far as real-valued potentials are considered, it comes
merely as a consequence of Sobolev inequalities that the
distance from the origin of every eigenvalue $z$ of the
Schr\"odinger operator lying in the negative semi-axis
can be bounded in terms of $L^p$ norm of the
potential, see~\cite{Keller,L_T,C_F_L}.
The non-self-adjoint situation requires
different tools. A key strategy in the subject was provided
by Abramov, Aslanyan and Davies: in~\cite{A_A_D} they
prove that for a possibly complex-valued $V,$ every discrete eigenvalue $z\in [0,\infty)$ of the
one-dimensional Schr\"odinger operator $-d^2/dx^2 + V$  lies in the complex plane within
a $1/4\|V\|_{L^1(\R)}^2$ distance from the origin. The generalization
to the higher dimensional case $d\geq 2$ was developed in a series of
work by different authors~\cite{F_L_L_S,La_Sa,Sa,Frank,Fr_Sa17}, just
to cite some among several relevant contributions. Eventually, Frank
and Simon~\cite{FS17}, using suitable resolvent estimates by Kenig,
Ruiz and Sogge~\cite{K_R_S}, proved the validity
of bounds of type~\eqref{eq:Lp-bound} for any $d\geq 2$ and any eigenvalue $z\in \C$
of the Schr\"odinger operator $-\Delta + V,$ with short-range potentials
$V\in L^{\gamma+ d/2}(\R^d),$ $\gamma\leq 1/2.$ In the same
work~\cite{FS17} the authors investigated also the case of long-range
potentials and showed that a bound of the form~\eqref{eq:Lp-bound}
could not hold for such a class: they construct
a sequence of real-valued potentials $V_n$ with
$\|V_n\|_{L^{\gamma+d/2}(\R^d)}\to 0,$ $\gamma>1/2,$ such that
$-\Delta + V_n$ has eigenvalue 1. A better understanding of the
distribution of eigenvalues of Schr\"odinger operators with slowly
decaying potentials $V\in L^{\gamma+d/2}(\R^d),$ $\gamma>1/2,$ was led
later by Enblom~\cite{Enblom16} and
Frank~\cite{FrankIII}. In~\cite{FrankIII} it is proved that a bound of
type \eqref{eq:Lp-bound} holds true with a correction which depends on the distance of the eigenvalue $z$ from the positive half-line, that is, defining $\delta(z):=\dist(z, [0,\infty)),$ one has
\begin{equation}\label{eq:alternative-bound}
  \delta(z)^{\gamma-1/2}|z|^{1/2}\leq C_{\gamma,\delta}\|V\|_{L^{\gamma+d/2}(\R^d)}^{\gamma+d/2}.
\end{equation}
Notice that~\eqref{eq:alternative-bound} is weaker than~\eqref{eq:Lp-bound} since $\delta(z)\leq |z|.$
As far as the size of the potential is measured in terms of
$L^p$ norms, one requires $p\geq d/2$ if $d\geq 3$ and $p> 1$
if $d=2$ in order to define $-\Delta + V$ as an $m$-sectorial
operator: this rules out the possibility
to treat physically interesting classes of potentials  which
might display stronger local singularities 
and demands for enlarging the class of potentials
considered. The analogous of bound~\eqref{eq:MC-bound} for 
Schr\"odinger operators with potentials in the Morrey-Campanato
class can be found in~\cite{Frank}, whereas
the analogous of bound~\eqref{eq:KS-bound} for potentials in the Kerman-Saywer
class is proved by Lee and Seo in~\cite{Lee_Seo}, see
also~\cite{Seo}. We observe that the bound
obtained in~\cite{Lee_Seo} presents a constant which is
independent of $V$, differently
from~\eqref{eq:KS-bound} in our setting: this shows a pathological behavior of the Lamé
operator as compared to the Schr\"odinger operator, consequence of the 
non-uniform weighted boundedness properties of the Riesz
transform with respect to  the weight $|V|$, \emph{cfr.} \Cref{boundedness-Riesz-transform}.  

The proofs of bounds~\eqref{eq:Lp-bound}, \eqref{eq:MC-bound} and \eqref{eq:KS-bound}
(\emph{cfr.}~\cite[Theorems~1.2--1.4]{Co19}) all display the same
underlying structure strongly based on the Birman-Schwinger principle
(\emph{cfr}.~\cite{Simon71}, Thm.~III.12, Thm.~III.14). The usefulness
of the Birman-Schwinger principle to localize eigenvalues of self-adjoint
and non-self-adjoint Hamiltonians is by no means questionable, as a
matter of fact an extensive bibliography on the subject has been
produced adopting this methodology. Without any hope of completeness
we refer to~\cite{Frank,FS17,FKV18} for results on Schr\"odinger
operators and~\cite{Ib_St19} for an adaptation to the discrete
setting, see also~\cite{Kr_Ku20} where matrix-valued damped wave
operators are concerned. Lower order operators, such as Dirac or fractional
Schr\"odinger models, are investigated
in~\cite{Cu_La_Tr14,Cuenin,FK19,Ca_Pi_Ve20,DA_Fa_Sc20} (see
also~\cite{Fe_La_Sa19,Cu19}) and in~\cite{Ca_Ib_Kr_St19} respectively
in the continuous and discrete scenario; as for higher order operators
refer to~\cite{IKL19}. Associated spectral stability results obtained with different techniques and related tools can be found in~\cite{Bu_D'A_Fa, Cacciafesta, Fanelli,Ca_D'A,CFK,A-L_K_H,F_L20}.   

In our context, the Birman-Schwinger principle
states that $z\in \C \setminus [0,\infty)$ is an eigenvalue of
$-\Delta^\ast + V$ \emph{if and only if} $-1$ is an eigenvalue of the
Birman-Schwinger operator $K_z:=|V|^{1/2}(-\Delta^\ast
-z)^{-1}V_{1/2}$ on $L^2(\R^d)^d,$ where $V_{1/2}:=|V|^{1/2} \sgn(V)$
and $\sgn(V)$ denotes the complex sign function. In particular, if
$-1$ is an eigenvalue of $K_z$ the norm of $K_z$ is at least one and
then proving bounds~\eqref{eq:Lp-bound}, \eqref{eq:MC-bound} and \eqref{eq:KS-bound} descends
from proving that 
\begin{equation*}
	\||V|^{1/2}(-\Delta^\ast -z)^{-1}V_{1/2}\|_{L^2\to L^2}^{\gamma + \frac{d}{2}}\leq c |z|^{-\gamma}\|V\|^{\gamma + \frac{d}{2}},
\end{equation*}    
where $\|V\|=\|V\|_{L^{\gamma + \frac{d}{2}}(\R^d)}$,
$\|V\|=\|V\|_{\mathcal{L}^{\alpha,p}(\R^d)}$ or
$\|V\|= Q_2(|V|) \| | V |^\beta\|_{\mathcal{KS}_\alpha(\R^d)}^{\frac{1}{\beta}}$
for~\eqref{eq:Lp-bound}, \eqref{eq:MC-bound} or~\eqref{eq:KS-bound}
respectively. Treating eigenvalues $z\in \C\setminus
[0,\infty)$, the Birman-Schwinger operator $K_z$ is well defined since $\sigma(-\Delta^\ast)=[0,\infty).$
The natural strategy to cover also $z\in [0,\infty)$ is to study an
approximating Birman-Schwinger operator, that is,
$K_{z+i\varepsilon}:=|V|^{1/2}(-\Delta^\ast
-z-i\varepsilon)^{-1}V_{1/2},$ for some $\varepsilon>0$, retracing the
proofs of~\eqref{eq:Lp-bound}, \eqref{eq:MC-bound} and~\eqref{eq:KS-bound} valid for $z+i\varepsilon$
 outside the spectrum and eventually passing to the limit
$\varepsilon\to 0.$ Thanks to this approach, in the following we extend 
\cite[Theorems 1.2--1.4]{Co19} to the whole point spectrum of
$-\Delta^\ast + V$.

The following theorem extends \cite[Theorems 1.2]{Co19} to treat the
whole point spectrum.
\begin{theorem}\label{thm:Lp-result}
	Let $d\geq 2,$ $0<\gamma\leq 1/2$ if $d=2$ and $0\leq \gamma\leq 1/2$ if $d\geq 3$ and $V\in L^{\gamma + \frac{d}{2}}(\R^d;\C^{d\times d}).$ Then there exists a universal constant $c_{\gamma, d, \lambda, \mu}>0$ independent on $V$ such that
	\begin{equation}\label{eq:thesis}
		\sigma_\textup{p}(-\Delta^\ast + V)
		\subset \left\{z\in \C\colon |z|^\gamma \leq c_{\gamma, d, \lambda, \mu} \| V \|_{L^{\gamma + \frac{d}{2}}(\R^d)}^{\gamma + \frac{d}{2}}\right \}. 
	\end{equation}
\end{theorem}
As a corollary, the previous theorem provides a sufficient condition
on the size of the potential to guarantee \emph{total} absence of
eigenvalues in the higher dimensional case $d\geq 3.$ 
\begin{corollary}\label{cor:absenceLp}
	If $d\geq 3$ and 
	\begin{equation*}
		c_{0, d, \lambda, \mu}\|V\|_{L^{\frac{d}{2}}(\R^d)}^{\frac{d}{2}}<1,
	\end{equation*}
	then $-\Delta^\ast + V$ has no eigenvalues.
  Furthermore, for $d=3$ the constant $c_{0,3,\lambda,\mu}$ is explicitly given by
  \begin{equation*}
    c_{0,3,\lambda,\mu}:=\left(\frac{2^{4/3}(1+ 6\cot^2(\pi/12))}{3\pi^{4/3}\min\{\mu, \lambda+ 2\mu\}}\right)^{\frac{3}{2}}.
  \end{equation*}	
\end{corollary}
\begin{remark}
In the context of Schr\"odinger operators, seeking for optimal
conditions on both local integrability and asymptotic decay of the
potentials under which absence of embedded eigenvalues is guaranteed
has yielded a considerable bibliography. Ionescu and Jerison
in~\cite{IJ2003} obtained absence of embedded eigenvalues for $V\in
L^{d/2}$ (or $V\in L^p,$ $p>1$ if $d=2$). We stress that as long as
local integrability conditions are investigated, this result is
optimal,  indeed Koch and Tataru in~\cite{KT2002} constructed non
trivial compactly supported solutions of the $0$-eigenvalue equation
$\Delta u=Vu$ with $V\in L^p_\textup{loc},$ $p<d/2$ for $d\geq 3$ (and
$V\in L^1_\textup{loc}$ for $d=2$). Later, Koch and Tataru
in~\cite{KT2006} proved the same result as in~\cite{IJ2003} for
potentials $V$ with the least possible decay at infinity, including
$V\in L^{(d+1)/2}.$ The exponent $(d+1)/2,$ $d\geq 2$ is the highest
possible, indeed Ionescu and Jerison~\cite{IJ2003} first and Frank and
Simon~\cite{FS17} later showed that there are operators with
potentials $V\in L^p,$ $p>(d+1)/2$ which admit positive
eigenvalues. Absence of embedded eigenvalues in the spirit
of~\cite{KT2002} for vector-valued Schr\"odinger operators was recently
obtained in~\cite{CM20}. In light of this remark, the constraint on $\gamma$ in~\Cref{thm:Lp-result} are rather natural.
\end{remark}

For potentials in the Morrey-Campanato class we prove the next
results, counterpart of \cite[Theorem 1.3, Corollary 1.1]{Co19}.
\begin{theorem}
\label{thm:MC-result}
	Let $d\geq 2,$  $(d-1)(2\gamma + d)/ 2(d-2\gamma)<p\leq \gamma + {d/2}$ with $0<\gamma\leq 1/2$ if $d=2$ and $0\leq \gamma\leq 1/2$ if $d\geq 3$ and assume $V\in \mathcal{L}^{\alpha,p}(\R^d;\C^{d\times d})$ with $\alpha=2d/(2\gamma+d).$ 	 
	Then there exists a universal constant $c_{\gamma, p, d, \lambda, \mu}>0$ independent on $V$ such that 
	\begin{equation*}
		\sigma_\textup{p}(-\Delta^\ast + V)
		\subset \left\{z\in \C\colon |z|^\gamma \leq c_{\gamma, p, d,\lambda, \mu} \| V \|_{\mathcal{L}^{\alpha,p}(\R^d)}^{\gamma + \frac{d}{2}}\right \}.
	\end{equation*}
\end{theorem}	
\begin{corollary}
	If $d\geq 3$ and 
	\begin{equation*}
		c_{0, p, d, \lambda, \mu}\|V\|_{\mathcal{L}^{2,p}(\R^d)}^{\frac{d}{2}}<1,
	\end{equation*}
	then $-\Delta^\ast + V$ has no eigenvalues.
	Furthermore, for $d=3$ the constant is explicitly given by
        \begin{equation*}
        c_{0,p,3,\lambda,\mu}:=\Big(\frac{c_\textup{F}(1+ 6C^2)}{\min\{\mu, \lambda+2\mu\}}\Big)^{\frac{3}{2}},
      \end{equation*}
with $c_\textup{F}=c_\textup{F}(p)$ as in~\eqref{eq:d=3-a} and  $C>0$ independent on $V.$
\end{corollary}
\begin{remark}
	We recall that, thanks to the H\"older inequality,
	\begin{equation*}
		\|V\|_{\mathcal{L}^{\alpha,p}(\R^d)}\leq \mathcal{V}_d^{\frac{1}{p}-\frac{\alpha}{d}} \|V\|_{L^\frac{d}{\alpha}(\R^d)},
	\end{equation*} 
	for $\alpha>0$ and $1\leq p \leq d/\alpha$ and where
        $\mathcal{V}_d$ denotes the volume of the unit $d$-dimensional
        ball. As a consequence, \Cref{thm:Lp-result} follows from
        \Cref{thm:MC-result} for $c_{\gamma,d,\lambda,\mu}$
        in~\eqref{eq:thesis} equal to $c_{\gamma,p,d,\lambda, \mu}
        (\mathcal{V}_{\alpha}^{1/p-\alpha/d})^{\gamma+ d/2},$ with
        $\alpha,$ $p$ and $c_{\gamma,p,d,\lambda, \mu}$ as in
        \Cref{thm:MC-result}. Nonetheless, we decided to state
        and also give an alternative proof of
        \Cref{thm:Lp-result} as it is of interest in its own
        right. As a matter of fact, in dimension $d=3$ and for $\gamma=0,$  this
        alternative direct proof
        provides an explicit bound on the constant
        $c_{\gamma,d,\lambda,\mu}$ in~\eqref{eq:thesis}
        and, in turn, on the smallness of the size of the potential in
        order to guarantee absence of eigenvalues. 
\end{remark}
Finally, the following theorem is the counterpart of \cite[Theorem
1.4]{Co19}, treating potentials in the Kerman-Saywer class.
\begin{theorem}\label{thm:KS-result}
	Let $d\geq 2,$ $1/3\leq \gamma<1/2$ if $d=2$ and $0\leq
        \gamma<1/2$ if $d\geq 3$ and assume
        $|V|^\beta\in\mathcal{KS}_{\alpha}(\R^d)$ with $\alpha=2d\beta
        (2\gamma +d)$ and $\beta=(d+2\gamma)(d-1)/[2(d-2\gamma)].$ If
        $|V|\in A_2(\R^d)$ then there exists a constant $c_{\gamma, d, \lambda, \mu}>0$ independent on $V$ such that 
	\begin{equation*}
		\sigma_\textup{p}(-\Delta^\ast + V)
		\subset \left\{z\in \C\colon |z|^\gamma \leq c_{\gamma, d, \lambda, \mu}\, Q_2(|V|)^{2\gamma +d} 
		\| |V|^\beta \|_{\mathcal{KS}_{\alpha}(\R^d)}^{\frac{1}{\beta}(\gamma +\frac{d}{2})}\right \}. 
	\end{equation*}
\end{theorem}	
\begin{corollary}\label{cor:KS}
	If $d\geq 3$ and 
	\begin{equation*}
		c_{0, d, \lambda, \mu}\, Q_2(|V|)^d \| |V| ^\frac{d-1}{2}\|_{\mathcal{KS}_{d-1}}^{\frac{d}{d-1}}<1,
	\end{equation*}
	then $-\Delta^\ast + V$ has no eigenvalues.
Furthermore, for $d=3$ the constant is explicitly given by
\begin{equation*}
c_{0,3,\lambda,\mu}:=\Big(\frac{c_{\textup{KS}}(1+6 C^2)}{\min\{\mu,\lambda+ 2\mu\}}\Big)^{\frac{3}{2}},
\end{equation*}
    with $c_{\textup{KS}}$ as in~\eqref{eq:d=3-a} and
 $C>0$ independent on $V.$ 		
\end{corollary}

 In the three dimensional case, 
the third author with Krej\v{c}i\v{r}\'ik and Vega proved in~\cite[Thm.~1]{FKV18} that the spectrum of
the three dimensional Schr\"odinger operator is stable under perturbations which satisfy the following subordination relation
\begin{equation}\label{eq:subodination-Schr}
	\exists\, a<1\quad \text{such that}	
	\quad \int_{\R^d} |V||u|^2\, dx\leq a \int_{\R^d} |\nabla u|^2\, dx,
	\qquad \forall\, u\in H^1(\R^d).
      \end{equation}
      In other words, under this assumption they not only show that
      the point spectrum is empty, but that the whole spectrum is
      absolutely continuous and equal to the spectrum of the
      unperturbed operator.
Their result relies on the proof of a variational one-sided version of
the conventional Birman-Schwinger principle extended to possible
eigenvalues embedded in the essential spectrum. This approach turned
out to be very robust, it was indeed adopted to investigate on the
spectrum of other Hamiltonians than the Schr\"odinger operators:
see~\cite{FK19} and~\cite{IKL19} for an adaptation to non-self-adjoint
Dirac and biharmonic operators, respectively.
We refer the reader to the recent work~\cite{Ha_Kr20} by Hansmann and Krej\v{c}i\v{r}\'ik for a systematic exposition of abstract Birman-Schwinger principles and their rigorous applications in spectral theory. 
To show the analogue result in the context of perturbed Lamé
operator, we need to recall that there exist $C>0$ such that for any $W \in
A_2(\R^3)$
the following sharp bound on the weighted $L^2$  operator norm of the Riesz transform
        $\mathcal{R}=(\mathcal{R}_1, \mathcal{R}_2,\dots,
        \mathcal{R}_d)$ is available, see \Cref{boundedness-Riesz-transform}:
        \begin{equation}\label{eq:see.lemma}
        \|\mathcal{R}_j\|_{L^2(W dx)\to L^2(W dx)}\leq c_W := C
        \,Q_2(W), \quad \text{ for all }j=1,\dots,d,
      \end{equation}
      where $Q_2(W)$ is the $0$--homogeneous $A_2$ constant of $W$ defined in
      \eqref{Q_p} below. 
\begin{theorem}\label{thm:dim3}
	Let $d=3.$ Assume that $V: \R^3 \to \C^{3\times 3}$, $|V|\in A_2(\R^3)$ and 
	\begin{equation}\label{eq:cond-FKV-d3}
		\exists\, a<\frac{\min\{\mu, \lambda+ 2\mu\}}{1+6c_V^2}\quad \text{such that}	
	\quad \int_{\R^3} |V||u|^2\, dx\leq a \int_{\R^3} |\nabla u|^2\, dx,
		\qquad \forall u\in H^1(\R^3),
	\end{equation} 
	with $c_V = c_{|V|}$ given by \eqref{eq:see.lemma}.
	Then $\sigma(-\Delta^\ast+V)=\sigma_\textup{c}(-\Delta^\ast+V)=[0,\infty).$
\end{theorem}
\begin{remark}
\label{rmk:d=3-KS}
Thanks to \eqref{eq:d=3-a}, a necessary and sufficient condition for the Hardy-type inequality
in~\eqref{eq:cond-FKV-d3} to hold is that $V$ belongs to the
Kerman-Saywer class $\mathcal{KS}_{2}(\R^3);$ furthermore in this
case, $a=c_\textup{KS} \|V\|_{\mathcal{KS}_2(\R^3)}$.
This entails that the subordination condition~\eqref{eq:cond-FKV-d3} holds if and only if $V\in \mathcal{KS}_{2}(\R^3)$ and $c_\textup{KS} \|V\|_{\mathcal{KS}_2(\R^3)}<\min\{\mu, \lambda+ 2\mu\}/(1+6c_V^2).$ We decided anyway to state~\Cref{thm:dim3} with the smallness condition~\eqref{eq:cond-FKV-d3} instead of requiring smallness of the Kerman-Sawyer norm of $V$ in line with what just pointed out, in order to keep with the subordination relation~\eqref{eq:subodination-Schr} introduced in~\cite{FKV18}.
\end{remark}

\begin{remark}
  In order to define
  $-\Delta^\ast + V$ as an m-sectorial operator it
  is sufficient to assume $a<\min\{\mu,
  \lambda+ 2\mu\}$, see \Cref{preliminaries}. The stronger condition
  on $a$ in~\eqref{eq:cond-FKV-d3} is needed to ensure the boundedness of
  the Birman-Schwinger operator
  $K_z:=|V|^{1/2}(-\Delta^\ast-z)^{-1}V_{1/2}$ with bound strictly
  less than one. We stress that the demand for the stronger smallness
  condition in~\eqref{eq:cond-FKV-d3} is connected to the elasticity
  framework of the Lamé operator and the need of the
  Helmholtz decomposition, refer to the proof of
  Lemma~\ref{lemma:bdd-K_z}. On the other hand, for the
  Birman-Schwinger operator associated to the Laplacian 
  $K_z^{\Delta}:=|V|^{1/2}(-\Delta -z)^{-1}V_{1/2},$ the validity
  of~\eqref{eq:subodination-Schr} directly gives $\|K_z^{\Delta}\|\leq
  a<1,$ with $a$ the same constant as in~\eqref{eq:subodination-Schr},
  refer to \cite[Lemma 1]{FKV18}. 
\end{remark}

In the following theorem we obtain the spectral stability stated in
\Cref{thm:dim3} in the case that $V$ belongs to the Morrey-Campanato class $\mathcal{L}^{2,p}(\R^3),$ $1<p\leq 3/2$. Notice that $\mathcal{L}^{2,3/2}(\R^3)=L^{3/2}(\R^3)$ is also covered.
\begin{theorem}
\label{thm:M-Cd3}
	Let $d=3.$ Assume $V\in \mathcal{L}^{2,p}(\R^3),$ $1<p\leq 3/2$ and
	\begin{equation}\label{eq:cond-d3-M-C}
          \frac{c_\textup{F}(1+ 6c_V^2)}{\min\{\mu,\lambda+2\mu\}}
          \|V\|_{\mathcal{L}^{2,p}(\R^3)}<1,
	\end{equation} 
	with $c_\textup{F}$ as in~\eqref{eq:d=3-a} and $c_V=c_{|V|}$ given by~\eqref{eq:see.lemma}. Then $\sigma(-\Delta^\ast+V)=\sigma_\textup{c}(-\Delta^\ast+V)=[0,\infty).$
\end{theorem}
\begin{remark}
        In \Cref{thm:M-Cd3} the assumption that the potential $V$ is
        in the Morrey-Campanato class allows to drop the assumption
        that it belongs to $A_2(\R^3)$. Thanks to Lemma~\ref{lemma:MC-M},
        if $V$ is in the Morrey-Campanato class, $Q_2(V)$ can be
        bounded by a constant independent on $V;$ in turn from~\eqref{eq:see.lemma} one has that $c_V$ in~\eqref{eq:cond-d3-M-C} is independent on $V$ too.
\end{remark}
\begin{remark}	
  Notice that the smallness condition~\eqref{eq:cond-d3-M-C} ensures $\|K_z\|<1$
  being $K_z$ the Birman-Schwinger operator (\emph{cfr.}~\eqref{BS-MC-d=3} in~\Cref{lemma_BS-estimates-d=3}). From~\eqref{eq:cond-d3-M-C} one has in particular that $c_\textup{F}\|V\|_{\mathcal{L}^{2,p}(\R^3)}<\min\{\mu,\lambda+2\mu\},$ hence $-\Delta^\ast +V$ is well defined as an m-sectorial operator due to~\eqref{eq:Hardy} and~\eqref{eq:d=3-a} (see also Section~\ref{preliminaries}).
\end{remark} 

Even though the case of $V\in L^{3/2}(\R^3)$ is covered by the
previous result
($p=3/2$ is allowed in \Cref{thm:M-Cd3}),
if we restrict to the Lebesgue setting we are able to prove
a more explicit result than~\Cref{thm:M-Cd3}.
This comes from the availability in the $L^p$ framework of the Hardy-Littlewood-Sobolev inequality (see proof of~\eqref{BS-Lp-d=3} in Lemma~\ref{lemma_BS-estimates-d=3}). 
	More specifically, we can prove the following $L^p$ framed result. 
	\begin{theorem}\label{thm:L^p-d3}
	Let $d=3.$ Assume $V\in L^{3/2}(\R^3).$ If 
	\begin{equation}\label{eq:cond-d3-Lp}
		\frac{2^{4/3}(1+
                  6\cot^2(\pi/12))}{3\pi^{4/3}\min\{\mu,
                  \lambda+2\mu\}} \|V\|_{L^{3/2}(\R^3)}<1,
	\end{equation} 
	then $\sigma(-\Delta^\ast+V)=\sigma_\textup{c}(-\Delta^\ast+V)=[0,\infty).$
\end{theorem}

\begin{remark}
Observe that the smallness condition~\eqref{eq:cond-d3-Lp} ensures $\|K_z\|<1$
(\emph{cfr.}~\eqref{BS-Lp-d=3} in~\Cref{lemma_BS-estimates-d=3}).
Furthermore notice that it follows from H\"older inequality and Sobolev embedding that
\begin{equation}
  \label{eq:Sob-Hol}
  \int_{\R^3} |V||u|^2\, dx
  \leq \|V\|_{L^{3/2}(\R^3)} \|u\|_{L^6(\R^3)}^2
  \leq \frac{2^{4/3}}{3 \pi^{4/3}} \|V\|_{L^{3/2}(\R^3)} \int_{\R^3} |\nabla u|^2\, dx.
\end{equation} 
From~\eqref{eq:cond-d3-Lp} one has in particular that $2^{4/3}/(3\pi^{4/3}) \|V\|_{L^{3/2}(\R^3)}<\min\{\mu,\lambda+2\mu\},$ hence $-\Delta^\ast +V$ is well defined as an m-sectorial operator.
\end{remark}
\begin{remark}[Comparison between Corollaries~\ref{cor:absenceLp}--\ref{cor:KS} and Theorems~\ref{thm:dim3}--\ref{thm:L^p-d3}]
	Observe that both Corollaries~\ref{cor:absenceLp}--\ref{cor:KS} and Theorems~\ref{thm:dim3}--\ref{thm:L^p-d3} are providing with sufficient smallness-type conditions on the perturbation $V$ which ensure stability (in an appropriate sense) of the spectrum of the free Lamé operator. Notice that if on one hand Corollaries~\ref{cor:absenceLp}--\ref{cor:KS} seem more general as they are stated for any dimension $d\geq 3$ (whereas Theorems~\ref{thm:dim3}--\ref{thm:L^p-d3} are valid in $d=3$ only), on the other hand Theorems~\ref{thm:dim3}--\ref{thm:L^p-d3} give a more complete description of the spectrum of the perturbed Hamiltonian ensuring the full stability $\sigma(-\Delta^\ast + V)=[0,\infty)=\sigma(-\Delta^\ast)$ instead of stability of the sole point spectrum, \emph{i.e.} $\sigma_\textup{p}(-\Delta^\ast + V)=\varnothing=\sigma_\textup{p}(-\Delta^\ast),$ as in Corollaries~\ref{cor:absenceLp}--\ref{cor:KS}. Nevertheless, as far as the case $d=3$ is considered and if we focus on the point spectrum only, then Corollaries~\ref{cor:absenceLp}--\ref{cor:KS} and Theorems~\ref{thm:dim3}--\ref{thm:L^p-d3} equal one another, in the sense that they provide the \emph{same} smallness conditions on the potentials to guarantee the stated stability.  
\end{remark}

The rest of the paper is organized as follows: in \Cref{preliminaries}
we collect some preliminary results related to the Lamé operator which
will be used later in the paper. The proofs of
Theorems~\ref{thm:Lp-result}--\ref{thm:KS-result} are provided in
\Cref{sec:proofs}. \Cref{sec:proof-d3} is devoted to the proof of the
spectral stability valid in the three dimensional setting, namely
Theorems~\ref{thm:dim3}--\ref{thm:L^p-d3}. 
\subsection*{Notations}
\begin{itemize}
\item For $1\leq p<\infty$ and $u=(u_1,\dots,u_d) \in \C^d$ we denote
  $|u|_p:=(\sum_{j=1}^d |u_j|^p)^{1/p}$. Also, we will drop the
  subscript for $p=2$, writing $|u|:=|u|_2$.
\item For $u=(u_1, \dots, u_d) \in L^p(\R^d)^d$, we denote
  $\| u \|_{L^p(\R^d)^d} := \| |u(\cdot)|_p \|_{L^p(\R^d)}$
\item Let $V\colon \R^d \to \C^{d\times d}$; we define
\begin{equation*}
	|V(x)|_p:=\sup_{u\in \C^d} \frac{|V(x)u|_p}{|u|_p}, 
	\qquad \text{for a.a.}\; x\in \R^d,
\end{equation*}
and $\|V\|_{L^p(\R^d)^{d\times d}}:=\||V(\cdot)|_p\|_{L^p(\R^d)}.$
With a slight abuse of notation we consistently write
$\|V\|_{L^p(\R^d)}$ to indicate $\|V\|_{L^p(\R^d)^{d\times d}}.$
	\item As customarily, the notation $\langle \cdot, \cdot \rangle_{p p'}$ is used to denote the duality pairing $L^{p}\times L^{p'} \to \C,$ with $1/p+ 1/{p'}=1.$
\item Given a measurable function $w$, $L^p(wdx)$ stands for the
  $w$-weighted $L^p$ space on $\R^d$ with measure $w(x)dx.$
  	\item  Given a measurable non-negative function $w,$ we say that $w$ belongs to the $A_p(\R^d)$ Muckenhoupt class of weights, for $1<p<\infty$ if the following quantity
\begin{equation}\label{Q_p}
	Q_p(w)
	:=\sup_{Q} \Bigg ( \frac{1}{\abs{Q}} \int_Q w(x)\, dx \Bigg) \Bigg ( \frac{1}{\abs{Q}} \int_Q w(x)^{- \frac{1}{p-1}}\, dx \Bigg)^{p-1}
\end{equation}
is finite. Here the supremum is taken over any cube $Q$ in $\R^d.$
	\item Any given $u\in L^2(\R^d)^d$ can be decomposed as $u=u_S
          + u_P,$ where $u_S$ is a divergence-free vector field and
          $u_P$ is a gradient, see \Cref{lem:helmholtz}.          
 \item The Riesz transform $\mathcal{R}=(\mathcal{R}_1, \mathcal{R}_2,
   \dots, \mathcal{R}_d)$ is defined through the  Fourier transform by
	\begin{equation*}
		\widehat{R_j f}(\xi)=-i\frac{\xi_j}{|\xi|}\widehat{f}(\xi), 
		\qquad j=1,2,\dots, d, 
                \quad \text{ for all } f\in L^2(\R^d).
	\end{equation*}
\item Let $d=3$ and let $z\in \C \setminus [0,\infty).$ We denote by $\mathcal{G}_\zeta(x,y)$ the integral kernel associated to the resolvent of the Laplacian $(-\Delta-z)^{-1}.$ Its explicit expression is given 
	\begin{equation}\label{eq:integral-kernel}
		\mathcal{G}_\zeta(x,y):=\frac{1}{4\pi} \frac{e^{-\sqrt{-\zeta}|x-y|}}{|x-y|}.
	\end{equation} 	
	Here and in the sequel we choose the principal branch of the square root. 
\item We use the notations $C(\bullet), C_{\bullet}$ or $c(\bullet), c_\bullet$
  to emphasize the dependence of $C$ or $c$ on~$\bullet$. If not stated, these constants are in principle not explicit.

\end{itemize}

\subsection*{Acknowledgment}
L.C.~thanks O. Ibrogimov for manifesting an interest in the contents of the paper which strongly motivated this project. The authors are also grateful to D. Krej\v{c}i\v{r}\'ik and R. L. Frank for useful correspondences.

B.C.~is supported by Fondo Sociale Europeo – Programma
  Operativo Nazionale Ricerca e Innovazione 2014-2020,
  progetto PON: progetto AIM1892920-attivit\`a 2, linea 2.1.
The research of L.C.~is supported by the Deutsche Forschungsgemeinschaft (DFG) through CRC 1173.

\section{Preliminaries}\label{preliminaries}
In this section we collect some preliminary results on the Lamé
operator, referring to~\cite{Co19} and references therein for
more details.
We start recalling the Helmholtz decomposition of vector fields in
$L^2(\R^d)^d.$
\begin{lemma}[{\cite[Theorem 2.1, Lemma 2.2]{Co19}}]
\label{lem:helmholtz}
	Any vector field $u\in L^2(\R^d)^d$ can be uniquely split into its divergence-free (transversal) part $u_\textup{S}$ and its gradient-type (longitudinal) part $u_\textup{P}.$ More precisely, $u$ can be decomposed as follows
	\begin{equation*}
		u=u_\textup{S} + u_{\textup{P}},
	\end{equation*}
	with $u_\textup{P}=\nabla \varphi$ and $u_\textup{S}=u-u_\textup{P},$ where the scalar potential $\varphi$ satisfies $\Delta \varphi=\div u$ ensuring that $\div u_\textup{S}=0.$
        Furthermore the decomposition is $L^2$ orthogonal, that is 
	\begin{equation*}
		\|u\|_{L^2(\R^d)}^2=\|u_\textup{S}\|_{L^2(\R^d)}^2+\|u_\textup{P}\|_{L^2(\R^d)}^2.
	\end{equation*}
	Moreover 
	\begin{equation}\label{eq:Helmholtz-Riesz}
		(\pi_\textup{S}u)_j=u_{\textup{S},j}
		=u_j + \sum_{k=1}^d \mathcal{R}_j\mathcal{R}_k u_k,
		\qquad \text{and} \qquad
		(\pi_\textup{P}u)_j=u_{\textup{P},j}
		=- \sum_{k=1}^d \mathcal{R}_j\mathcal{R}_k u_k,
		\qquad j=1,2,\dots, d,
	\end{equation} 
        being $\mathcal{R}=(\mathcal{R}_1, \dots, \mathcal{R}_d)$ is
        the Riesz transform.
\end{lemma}

The following two lemmas are easy consequence of the Helmholtz decomposition.
\begin{lemma}[{\cite[Lemma 2.2]{B_F-G_P-E_R_V}, \cite[Lemma 2.6]{Co19}}]
	Let $d\geq 2$ and let $f$ be a regular vector field sufficiently rapidly decaying at infinity. Then $-\Delta^\ast$ acts on $f=f_\textup{S}+f_\textup{P}$ as 
	\begin{equation}
	\label{eq:simple_writing}
		-\Delta^\ast f= -\mu \Delta f_\textup{S} -(\lambda +2 \mu) \Delta f_\textup{P},
	\end{equation}   
	where  $f_\textup{S}$ is a divergence free vector field and $f_\textup{P}$ a gradient.		
\end{lemma}

\begin{lemma}[{\cite[Lemma 2.2]{B_F-G_P-E_R_V}, \cite[Lemma 2.7]{Co19}}]
	Let $z\in \C\setminus [0, \infty)$ and $ g\in L^2(\R^d)^d.$ Then the identity 
	\begin{equation}
	\label{eq:resolvent}
		(-\Delta^\ast -z)^{-1} g= \frac{1}{\mu} \big(-\Delta - \tfrac{z}{\mu}\big)^{-1} g_\textup{S} + \frac{1}{\lambda + 2\mu} \big(-\Delta - \tfrac{z}{\lambda + 2\mu}\big)^{-1} g_\textup{P}
	\end{equation}
	holds true, where $g=g_\textup{S}+ g_\textup{P}$ is the Helmholtz decomposition of $g.$
\end{lemma}

Thanks to the representation~\eqref{eq:simple_writing} of
$-\Delta^\ast$ in terms of Laplace operators,
the quadratic form $h_0$ associated with $-\Delta^\ast$ has the following expression:
\begin{equation}\label{eq:h_0}
	h_0[u]:=\mu \int_{\R^d} |\nabla u_\textup{S}|^2\,dx + (\lambda + 2\mu) \int_{\R^d} |\nabla u_\textup{P}|^2\,dx,
	\qquad \mathcal{D}(h_0):=H^1(\R^d)^d.
\end{equation}
Let $V\colon \R^d \to \C^{d\times d}$ be a measurable matrix-valued function such that
\begin{equation*}
	\exists\, a<\min\{\mu, \lambda+ 2\mu\}
	\quad \text{such that}	
	\quad \int_{\R^d} |V||u|^2\, dx\leq a \int_{\R^d} |\nabla u|^2\, dx,
	\qquad \forall\, u\in H^1(\R^d),
\end{equation*}
 thus the quadratic form
\begin{equation}\label{eq:v}
	v[u]:=\int_{\R^d} \overline{Vu}\cdot u\, dx,
	\qquad \mathcal{D}(v):=\Big \{u\in L^2(\R^d)^d\colon \int_{\R^d} |V||u|^2\, dx<\infty\Big \}
\end{equation}
is relatively bounded with respect to $h_0$ with relative bound less than one. As a consequence, the sum $h_V:=h_0+v$ is a closed form with $\mathcal{D}(h_V)=H^1(\R^d)^d$ which gives rise to an m-sectorial operator in $L^2(\R^d)^d$ via the representation theorem (\emph{cf}.~\cite[Thm. VI.2.1]{Kato}).

In the following lemma we gather some boundedness results for the Riesz transform.
\begin{lemma}
  \label{boundedness-Riesz-transform}
  Let $1<p<\infty$ and $p'$ such that $1/p+1/p'=1$ and let $w
  \in A_p(\R^d)$. Then, for any $j=1,2,\dots, d,$
  the following bounds on the operator norms of the Riesz transform $\mathcal{R}_j$ hold true:
  \begin{equation}\label{Riesz_1}
    \norm{\mathcal{R}_j}_{L^p(\R^d) \to L^p(\R^d)}
    =c_p:=\cot\Big(\frac{\pi}{2 \max\{p,p'\}} \Big), 
  \end{equation}
  \begin{equation}\label{Riesz_2}
    \norm{\mathcal{R}_j}_{L^p(w dx)\to L^p(w dx)}\leq c_w:= C \, Q_p(w)^{\max\{1, p'/p\}}, 
  \end{equation}
  for some $C>0$ independent on $w$ and for $Q_p(w)$ defined as in \eqref{Q_p}.
  \begin{proof}
    For the proof of~\eqref{Riesz_1} refer to~\cite{B_W} (see also~\cite{Ca_Zy}); inequality~\eqref{Riesz_2} is proved in~\cite{Petermichl} (see also~\cite{C_F}).
        \end{proof}
	\end{lemma}
Thanks to \Cref{lem:helmholtz} and \Cref{boundedness-Riesz-transform}, one shows almost-orthogonality of the $S$ and $P$ components in the Helmholtz decomposition.
\begin{lemma}\label{lemma:orthogonality}
	Let $g=g_\textup{S} + g_\textup{P}$ be the Helmholtz decomposition of $g.$ For $1<p<\infty$ the following estimates hold true:
	\begin{align}\label{eq:orthogonality}
		 \|g_\textup{S}\|_{L^p(\R^d)^d} + \|g_\textup{P}\|_{L^p(\R^d)^d}
		&  \leq (1+2 d c_p^2)\|g\|_{L^p(\R^d)^d},
          \\
          \label{eq:orthogonality-2}
		 \|g_\textup{S}\|_{L^2(w dx)^d} + \|g_\textup{P}\|_{L^2(wdx)^d}
		& \leq (1+2 d c_w^2)\|g\|_{L^2(wdx)^d},
	\end{align}
	with $c_p,c_w>0$ defined in
        \Cref{boundedness-Riesz-transform}. 
\end{lemma} 
\begin{proof}
  We prove only~\eqref{eq:orthogonality}, the proof
  of~\eqref{eq:orthogonality-2} is similar.
  Let $g \in L^p(\R^d)^d$: from~\eqref{eq:Helmholtz-Riesz} one has
  \begin{equation*}
  	\|g_\textup{S}\|_{L^p(\R^d)^d} + \|g_\textup{P}\|_{L^p(\R^d)^d}
  	\leq \|g\|_{L^p(\R^d)^d} + 2 \|\textstyle{\sum_{k=1}^d} \mathcal{R}\mathcal{R}_kg_k\|_{L^p(\R^d)^d}.
  \end{equation*} 
  Using~\eqref{Riesz_1} one gets
  \begin{equation*}
  	\begin{split}
  		\|\textstyle{\sum_{k=1}^d} \mathcal{R}\mathcal{R}_kg_k\|_{L^p(\R^d)^d}
  		&\leq \textstyle{\sum_{k=1}^d}
                (\textstyle{\sum_{j=1}^d} \|\mathcal{R}_j
                \mathcal{R}_kg_k\|_{L^p(\R^d)}^p)^\frac{1}{p}\\
  		&\leq d^\frac{1}{p} c_p^2  \textstyle{\sum_{k=1}^d} \|g_k\|_{L^p(\R^d)}\\
  	&\leq d c_p^2 \|g\|_{L^p(\R^d)^d},
	\end{split}  
  \end{equation*} 
  where in the last inequality we have used the Hölder inequality for discrete measure. Gathering the two previous bounds gives~\eqref{eq:orthogonality}. 
\end{proof}

The following result shows the good behavior of the Morrey-Campanato space in relation with the Muchenhoupt class of weights.
\begin{lemma}[{\cite[Lemma 1]{Ch_Fr}}]\label{lemma:MC-M}
  Let $0<\alpha<d$, $1<p\leq d/\alpha$ and let $V \in
  \mathcal{L}^{\alpha, p}(\R^d)$, $V\geq 0$.
  If $p_1 \in (1,p),$ then $W=(MV^{p_1})^{1/p_1} \in A_1(\R^d)\cap
  \mathcal{L}^{\alpha,p}(\R^d),$ where $M$ denotes the usual
  Hardy-Littlewood maximal operator, and $V(x)\leq W(x)$ for almost all $x\in \R^d.$
  Moreover, there exists a constant $C>0$ independent on $V,$ such the $A_1$ constant for $W$ is less than $C$ and 
	\begin{equation}\label{eq:W-V}
		\|W\|_{\mathcal{L}^{\alpha,p}(\R^d)}\leq C\|V\|_{\mathcal{L}^{\alpha,p}(\R^d)}.
	\end{equation}
\end{lemma}

In the next lemma we list uniform estimates for the
operator norm of the resolvent $(-\Delta^\ast -z)^{-1},$ $z\in
\C\setminus [0,\infty).$ 
\begin{lemma}
  \label{lemma:uniform-estimate}	
  Let $z\in \C\setminus [0,\infty).$ Then the following estimates for the resolvent $(-\Delta^\ast-z)^{-1}$ hold true.
  \begin{enumerate}[i)]
  \item Let $1<p\leq 6/5$ if $d=2,$ $2d/(d+2)\leq p \leq
    2(d+1)/(d+3)$ if $d\geq 3$ and let $p'$ such that
    $1/p + 1/p'=1.$ Then there exists a universal constant
    $c_{p,d,\lambda, \mu}>0$ such that 
    \begin{equation}\label{Lame_res_1}
      \norm{(-\Delta^\ast-z)^{-1}}_{L^p(\R^d)\to L^{p'}(\R^d)}
      \leq c_{p,d,\lambda, \mu}\abs{z}^{-\frac{d+2}{2} + \frac{d}{p}}.
    \end{equation}
  \item Let $3/2< \alpha < 2$ if $d=2,$ $2d/(d+1)<\alpha \leq 2$ if
    $d\geq 3$ and let $(d-1)/2(\alpha-1)< p \leq d/\alpha.$ Then there
    exists a universal constant $c_{\alpha, p,d,\lambda, \mu}>0$ such
    that for any non-negative function $V$ in $\mathcal{L}^{\alpha, p}(\R^d)$
    \begin{equation}\label{Lame_res_2}
      \norm{(-\Delta^\ast-z)^{-1}}_{L^2(V^{-1}dx)\to L^{2}(Vdx)}
      \leq c_{\alpha, p,d,\lambda, \mu} \abs{z}^{-1+ \frac{\alpha}{2}}\norm{V}_{\mathcal{L}^{\alpha, p}(\R^d)}.
    \end{equation}
  \item Let $3/2\leq \alpha < 2$ if $d=2,$ $d-1\leq \alpha < d$ if
    $d\geq 3$ and let $\beta=(2\alpha -d+1)/2.$
    Then there exists a universal constant $c_{\alpha,
      d,\lambda,\mu}>0$ such that for any non-negative function $V$ such that $\abs{V}^\beta\in \mathcal{KS}_{\alpha}(\R^d)$
    \begin{equation}\label{Lame_res_3}
      \norm{(-\Delta^\ast-z)^{-1}}_{L^2(V^{-1}dx)\to
        L^{2}(Vdx)}
      \leq  c_{\alpha, d,\lambda,\mu} \, Q_2(V)^2
    \abs{z}^{-\frac{\alpha - d+1}{2\alpha -d+1}}
    \norm{\abs{V}^\beta}_{\mathcal{KS}_{\alpha}(\R^d)}^\frac{1}{\beta}. 
  \end{equation}
\end{enumerate}
\begin{proof}
  For a proof refer to~\cite[Thm.~2.3]{Co19} (see also~\cite[Thm.~1.1]{B_F-G_P-E_R_V}).
  These estimates are consequence of the corresponding bounds for the
  resolvent of the free Schr\"odinger operator $(-\Delta-z)^{-1}$
  (\emph{cfr.}~\cite[Thm.~2.2]{Co19}
  and~\cite[Thm.~3.8]{B_F-G_P-E_R_V} for the collection of statements
  and references therein for the explicit proof), after using the explicit representation~\eqref{eq:resolvent} and the boundedness properties of the Riesz transform~\ref{boundedness-Riesz-transform}.	\end{proof}
\end{lemma}

From \Cref{lemma:uniform-estimate}, the following estimates descend for the
Birman-Schwinger operator associated to the Lamé operator.
\begin{lemma}\label{lemma_BS-estimates}
  Let $z\in \C \setminus [0,\infty),$ $0<\gamma\leq 1/2$ if $d=2$ and $0\leq \gamma\leq 1/2$ if $d\geq 3.$  Then the following estimate for the $L^2-L^2$ operator norm of the Birman-Schwinger operator $K_z:=|V|^{1/2}(-\Delta^\ast-z)^{-1}V_{1/2}$ hold true:
		\begin{equation}
			\label{BS-Lp}
				\|K_z\|_{L^2(\R^d)\to L^{2}(\R^d)}\leq c_{\gamma, d,\lambda, \mu}\abs{z}^{-\frac{2\gamma}{2\gamma +d}} \|V\|_{L^{\gamma + \frac{d}{2}}(\R^d)}.
		\end{equation}
                For $p$ and $\alpha$ as in \Cref{thm:MC-result} one has
		\begin{equation}
			\label{BS-MC}
				\|K_z\|_{L^2(\R^d)\to L^{2}(\R^d)}\leq c_{\gamma, p, d,\lambda, \mu}\abs{z}^{-\frac{2\gamma}{2\gamma +d}}
			\|V\|_{\mathcal{L}^{\alpha, p}(\R^d)}.
		\end{equation}
		If, in addition, $V\in A_2(\R^d),$ then for $\alpha$ and $\beta$ as in \Cref{thm:KS-result} one has
		\begin{equation}\label{BS-KS}
                  \|K_z\|_{L^2(\R^d)\to L^{2}(\R^d)}
                  \leq c_{\gamma, d,\lambda, \mu}\, Q_2(|V|)^2
                  \abs{z}^{-\frac{2\gamma}{2\gamma +d}}
                  \| |V|^\beta\|_{\mathcal{KS}_{\alpha}(\R^d)}^\frac{1}{\beta}.
		\end{equation}
\end{lemma}
\begin{proof}
	The three estimates~\eqref{BS-Lp},~\eqref{BS-MC} and~\eqref{BS-KS} follow straightforwardly from the validity of~\eqref{Lame_res_1},~\eqref{Lame_res_2} and~\eqref{Lame_res_3}, respectively. An explicit proof is obtained in~\cite{Co19} as a byproduct of the proofs of Theorem 1.2, Theorem 1.3 and Theorem 1.4 there.
\end{proof}

For later purposes, we rewrite the statement of \Cref{lemma_BS-estimates} in
the case that $d=3$ and $\gamma=0$, since in this situation we are able to provide more explicit information on the bound of the operator norm of the Birman-Schwinger operator.
\begin{lemma}\label{lemma_BS-estimates-d=3}
  Let $d=3$ and let $z\in \C \setminus [0,\infty).$  Assume $1<p\leq
  3/2.$  Then the following estimates for the Birman-Schwinger operator $K_z:=|V|^{1/2}(-\Delta^\ast-z)^{-1}V_{1/2}$ hold true:
  \begin{align}
    \label{BS-Lp-d=3}
    & \|K_z\|_{L^2(\R^3)\to L^{2}(\R^3)}\leq \frac{2^{4/3}(1+6\cot^2(\pi/12))}{3\pi^{4/3} \min\{\mu, \lambda+2\mu\}} \|V\|_{L^{3/2}(\R^3)},
    \\
    & \label{BS-MC-d=3}
      \|K_z\|_{L^2(\R^3)\to L^{2}(\R^3)}\leq
      \frac{c_\textup{F}(1+ 6 C^2)}{\min\{\mu,\lambda+2\mu\}}
      \|V\|_{\mathcal{L}^{2, p}(\R^3)},
  \end{align}
    with $c_\textup{F}$ as in \eqref{eq:d=3-a} and  $C>0$.
    If, in addition, $V\in A_2(\R^3),$ then
    \begin{equation}\label{BS-KS-d=3}
      \|K_z\|_{L^2(\R^3)\to L^{2}(\R^3)}\leq
      \frac{c_{\textup{KS}}(1+6c_V^2)}{\min\{\mu,\lambda+ 2\mu\}}
      \|V\|_{\mathcal{KS}_2(\R^3)},
    \end{equation}
    with $c_{\textup{KS}}$ as in~\eqref{eq:d=3-a} and
    $c_V :=C\, Q_2(|V|),$ for $C>0$. 
\begin{proof}
  Bounds~\eqref{BS-Lp-d=3}-\eqref{BS-KS-d=3} are
  simply bounds~\eqref{BS-Lp}-\eqref{BS-KS} in
  the specific framework considered here. To get the explicit values of the constants in
  in~\eqref{BS-Lp-d=3}-\eqref{BS-KS-d=3}, 
  we provide a direct proof which relies on the explicit expression of the integral kernel $\mathcal{G}_z(x,y)$ of $(-\Delta-z)^{-1}$ in $d=3.$
			
To bound the operator norm of $K_z$ we estimate the inner product $\langle f, K_z g \rangle,$ for any $f, g\in L^2(\R^3)^3.$ The relation~\eqref{eq:resolvent} gives
\begin{equation}\label{eq:initial}
			|\langle f, K_z g \rangle|
			\leq\frac{1}{\mu}|\langle f, |V|^{1/2}(-\Delta - \tfrac{z}{\mu}) G_\textup{S} \rangle| + \frac{1}{\lambda + 2\mu}|\langle f, |V|^{1/2}(-\Delta -\tfrac{z}{\lambda + 2\mu}) G_\textup{P} \rangle|,
\end{equation}
where we set $G =G_\textup{S} + G_\textup{P}:=V_{1/2}g.$
			First we estimate $|\langle f, |V|^{1/2}(-\Delta - \tfrac{z}{\mu}) G_\textup{S} \rangle|.$
			Given the explicit expression~\eqref{eq:integral-kernel} for the integral kernel of $(-\Delta-\zeta)^{-1},$ $\zeta\in \C\setminus [0,\infty),$ one has that $\mathcal{G}_\zeta(x,y)$ is bounded in absolute value by the Green function $\mathcal{G}_0(x,y):=(4\pi |x-y|)^{-1},$ \emph{i.e.}, $|\mathcal{G}_\zeta(x,y)|\leq \mathcal{G}_0(x,y).$ Hence
\begin{equation*}
  \begin{split}
    |\langle f, |V|^{1/2}(-\Delta - \tfrac{z}{\mu}) G_\textup{S} \rangle|
    &\leq\langle |f|, |V|^{1/2}|(-\Delta-\tfrac{z}{\lambda + 2\mu})G_\textup{S}| \rangle\\
    &=\iint_{\R^3\times \R^3} |f|(x)|V(x)|^{1/2} |\mathcal{G}_{z/\mu} (x,y)| |G_\textup{S}(y)|\, dx\, dy\\
    &\leq \frac{1}{4\pi} \iint_{\R^3\times \R^3} \frac{|f(x)| \, |V(x)|^{1/2} |G_\textup{S}(y)|}{|x-y|}\, dx\, dy\\
    &\leq \frac{2^{4/3}}{3\pi^{4/3}} \||f|\,|V|^{1/2}\|_{L^{6/5}(\R^3)} \||G_\textup{S}|\|_{L^{6/5}(\R^3)},
  \end{split}
\end{equation*}
where in the last inequality we used the sharp
Hardy-Littlewood-Sobolev inequality (see~\cite{Li1983}, \cite[Thm.~4.3]{Li_Lo}). Analogous computations for $|\langle f, |V|^{1/2}(-\Delta - \tfrac{z}{\mu}) G_\textup{P} \rangle|$ give
			\begin{equation*}
				|\langle f, |V|^{1/2}(-\Delta - \tfrac{z}{\mu}) G_\textup{P} \rangle|
				\leq \frac{2^{4/3}}{3\pi^{4/3}} \||f|\,|V|^{1/2}\|_{L^{6/5}(\R^3)} \||G_\textup{P}|\|_{L^{6/5}(\R^3)}.
			\end{equation*}
			 Plugging the last estimates in the bound~\eqref{eq:initial} gives
			 \begin{equation*}
					|\langle f, K_z g \rangle|
					\leq \frac{1}{\min\{\mu, \lambda+2\mu\}}\frac{2^{4/3}}{3\pi^{4/3}}  \|f |V|^{1/2}\|_{L^{6/5}(\R^3)^3} \big(\|G_\textup{S}\|_{L^{6/5}(\R^3)^3} + \|G_\textup{P}\|_{L^{6/5}(\R^3)^3} \big).
			\end{equation*}
			Using the orthogonality property in
                        Lemma~\ref{lemma:orthogonality}, thanks to the H\"older inequality one has
			\begin{equation*}
				\begin{split}
	|\langle f, K_z g \rangle|
	&\leq \frac{1}{\min\{\mu, \lambda+2\mu\}}\frac{2^{4/3}}{3\pi^{4/3}} (1+ 6 \cot^2(\pi/12)) \|f |V|^{1/2}\|_{L^{6/5}(\R^3)^3} \|V_{1/2} \,g\|_{L^{6/5}(\R^3)^3}\\
	&\leq \frac{1}{\min\{\mu, \lambda+2\mu\}} \frac{2^{4/3}}{3\pi^{4/3}} (1+ 6 \cot^2(\pi/12)) \| V\|_{L^{3/2}(\R^3)^{3\times 3}} \|f\|_{L^2(\R^3)^3} \|g\|_{L^2(\R^3)^3}.
				\end{split}
			\end{equation*}
			Taking the supremum over all $f,g\in L^2(\R^3)^3$ with norm equal to one, gives the bound in~\eqref{BS-Lp-d=3}.
			
			Now we are in position to prove~\eqref{BS-MC-d=3} and~\eqref{BS-KS-d=3}. As a starting point we observe that under the assumptions of the lemma the Hardy type inequality~\eqref{eq:Hardy} with~\eqref{eq:d=3-a} holds true.	
			
It is known that estimates of type~\eqref{eq:Hardy} are equivalent to weighted
boundedness properties of the $1$- fractional integral operator
$I_1=H_0^{-1/2}$ where $H_0:=-\Delta$ (see~\cite[Lemma
1]{FKV18}). Then~\eqref{eq:Hardy} is equivalent to
	\begin{equation}\label{equiv}
		\||V|^{1/2}H_0^{-1/2}\|_{L^2\to L^2}\leq \sqrt{a}
	\end{equation}	
        and, by taking the adjoint, one also has
	\begin{equation}\label{adjoint}
		\|H_0^{-1/2}|V|^{1/2}\|_{L^2\to L^2}\leq \sqrt{a},
	\end{equation}
	where $a$ is as in~\eqref{eq:d=3-a}.
	
	Let us prove~\eqref{BS-MC-d=3} first, that is let us assume
        that $V\in \mathcal{L}^{2,p}(\R^3)$. By Lemma~\ref{lemma:MC-M}, there exists $W\in A_2(\R^3)\cap \mathcal{L}^{2,p}(\R^3)$ such that $V(x)\leq W(x)$ for almost all $x\in \R^3.$
	Using this fact and the pointwise bound
        $|\mathcal{G}_\zeta(x,y)|\leq \mathcal{G}_0(x,y),$ $z\in
        \C\setminus (0,\infty),\, x, y\in \R^3,$ we have
\begin{equation}\label{S}
  \begin{split}
    & |\langle f, |V|^{1/2}(-\Delta - \tfrac{z}{\mu}) G_\textup{S}
    \rangle|
    \\
    &\leq \langle |f|, |W|^{1/2}H_0^{-1} |G_{\textup{S}}| \rangle\\
    &\leq \|f\|_{L^2(\R^3)^3} \||W|^{1/2}H_0^{-1/2}\|_{L^2\to L^2} \|H_0^{-1/2}|W|^{1/2}\|_{L^2\to L^2} \|G_{\textup{S}}\|_{L^2(|W|^{-1}dx)^3}\\
    &\leq a \|f\|_{L^2(\R^3)^3} \|G_{\textup{S}}\|_{L^2(|W|^{-1}dx)^3},
  \end{split}
\end{equation}
where in the last inequality we have used bounds~\eqref{equiv} and~\eqref{adjoint}.
	Similar computations for the term in~\eqref{eq:initial} involving the $P$ component, namely $|\langle f, |V|^{1/2}(-\Delta - \tfrac{z}{\lambda + 2\mu}) G_\textup{P} \rangle|$ give
	\begin{equation}\label{P}
		|\langle f, |V|^{1/2}(-\Delta - \tfrac{z}{\lambda + 2\mu}) G_\textup{P} \rangle|\leq a \|f\|_{L^2(\R^3)^3}\|G_{\textup{P}}\|_{L^2(|W|^{-1}dx)^3}.
	\end{equation}
	Plugging~\eqref{S} and~\eqref{P} in~\eqref{eq:initial} one has
	\begin{equation*}
		|\langle f, K_z g \rangle|\leq \frac{a}{\min\{\mu, \lambda + 2\mu\}} \|f\|_{L^2(\R^3)^3} \big(\|G_{\textup{S}}\|_{L^2(|W|^{-1}dx)^3} + \|G_{\textup{S}}\|_{L^2(|W|^{-1}dx)^3} \big).
	\end{equation*}
	Using the orthogonality property~\eqref{eq:orthogonality-2} stated in Lemma~\ref{lemma:orthogonality}, recalling that $|W|^{-1}\leq |V|^{-1}$ almost everywhere and using that $G:=V_{1/2}g$ we get
	\begin{equation}\label{eq:last-BS}
		|\langle f, K_z g \rangle|\leq a \frac{1+ 6 C^2 Q_2(|W|)^2}{\min\{\mu, \lambda + 2\mu\}} \|f\|_{L^2(\R^3)^3}\|g\|_{L^2(\R^3)^3}.
	\end{equation}
	From Lemma~\ref{lemma:MC-M} we know that $Q_2(|W|)$ is less
        than a constant independent on $W$, so
        estimate~\eqref{BS-MC-d=3} follows from~\eqref{eq:last-BS}
        using the bound~\eqref{eq:W-V} in $a=c_\textup{F}\|W\|_{\mathcal{L}^{2,p}(\R^3)}$.
	
        The proof of~\eqref{BS-KS-d=3} descends
        from~\eqref{eq:last-BS} with $W=V$ and $a=c_\textup{KS}\|V\|_{\mathcal{KS}_{2}(\R^3)}.$
	\end{proof}	
\end{lemma}

The next lemma represents another relevant consequence of
  the boundedness of the Birman-Schwinger operator. 
  We are grateful to R.L.~Frank for showing us the argument.
\begin{lemma}\label{lemma:Frank}
Let $\gamma, p$ and $\alpha$ as in Lemma~\ref{lemma_BS-estimates}. If $V\in L^{\gamma+ d/2}(\R^d),$ $V\in \mathcal{L}^{\alpha, p}(\R^d),$ or $V\in \mathcal{KS}_\alpha(\R^d),$ then the multiplication by $|V|^{1/2}$ is a bounded operator from $H^1(\R^d)^d$ to $L^2(\R^d)^d.$
\end{lemma}
\begin{proof}
	Minor modifications of the argument in Lemma~\ref{lemma_BS-estimates} ensure that the operator $|V|^{1/2}(-\Delta^\ast -z)^{-1}|V|^{1/2}$ with $z\in (-\infty, 0]$ is a bounded operator in $L^2,$ more precisely
\begin{equation*}
	\||V|^{1/2}(-\Delta^\ast -z)^{-1} |V|^{1/2}\|\leq C(z,V) \|V\|,
\end{equation*}	
where $\|V\|$ denotes $\|V\|=\|V\|_{L^{\gamma + \frac{d}{2}}(\R^d)},$
$\|V\|=\|V\|_{\mathcal{L}^{\alpha, p}(\R^d)}$ or
$\|V\|=\|V\|_{\mathcal{KS}_\alpha(\R^d)}$ and $C(z,V)$ is a constant
that may depend on $|z|$ and $V$
(\emph{cfr.}~\eqref{BS-Lp}--\eqref{BS-KS}).
Since $z\in (-\infty,0]$, $(-\Delta^\ast-z)$ is a positive operator. 
We write
$|V|^{1/2}(-\Delta^\ast-z)^{-1}|V|^{1/2}=
AA^\ast$, with
$A:=|V|^2(-\Delta^\ast -z)^{-1/2}$. Using that
$\|AA^\ast\|=\|A^\ast\|^2=\|A\|^2=\||V|^{1/2}(-\Delta^\ast-z)^{-1/2}\|^2,$
 one has 
\begin{equation}\label{eq:rigorous}
	\begin{split}
		\int_{\R^d}
                |V||u|^2&=\||V|^{1/2}(-\Delta^\ast-z)^{-1/2}(-\Delta^\ast-z)^{1/2}
                u\|_{L^2(\R^d)}^2
                \\
		&\leq C(z, V)\|V\| \|(-\Delta^\ast-z)^{1/2}u\|_{L^2(\R^d)}^2\\
		&=C(z, V)\|V\| \langle u, (-\Delta^\ast-z) u \rangle\\
		&=C(z, V)\|V\| \big(h_0[u] -z\|u\|_{L^2(\R^d)}^2\big),
	\end{split}
\end{equation}  
where $h_0$ denotes the quadratic form associated to the Lamé operator $-\Delta^\ast$ defined in~\eqref{eq:h_0}. 
Using the explicit expression~\eqref{eq:h_0} for the quadratic form $h_0$ one can rewrite~\eqref{eq:rigorous} as
\begin{equation}\label{eq:generalization}
	\int_{\R^d}|V||u|^2\leq C(z,V, \lambda, \mu)\|V\|\big(\|\nabla u\|_{L^2(\R^d)}^2 - z\|u\|_{L^2(\R^d)}^2\big),
\end{equation}
where $C(z,V, \lambda, \mu)=C(z,V)\max\{\mu, \lambda+2\mu\}$ and $C(z, V)$ is as in~\eqref{eq:rigorous}. Hence, $|V|^{1/2}u\in L^2(\R^d)^d$ whenever $u\in H^1(\R^d)^d.$ 
\end{proof}

\begin{remark}
	Assuming an Hardy-type condition like
\begin{equation}\label{eq:Hardy-type} 
	\int_{\R^d} |V||u|^2\,dx\leq C(V)\int_{\R^d} |\nabla u|^2\, dx, 
	\qquad \forall u\in C^\infty_0(\R^d)
\end{equation}
would also serve the purpose of ensuring boundedness of $|V|^{1/2}$ as an operator from $H^1(\R^d)^d$ to $L^2(\R^d)^d.$ As a matter of fact, it is known that~\eqref{eq:Hardy-type} holds true if $V\in L^{d/2}(\R^d),$ that is $V\in L^{\gamma+ d/2}(\R^d)$ and $\gamma=0$ (as a consequence of H\"older inequality and Sobolev embedding), if $V\in \mathcal{L}^{\alpha,p}(\R^d)$ with $\alpha=2$ (which, in turn, gives $\gamma=0,$ recall $\alpha:=2d/(2\gamma+d)$) and $1<p\leq \frac{d}{2},$ fact that was discovered by Fefferman in~\cite{Fefferman1983} (see also Chiarenza-Frasca~\cite{Ch_Fr}) and if $V\in \mathcal{KS}_\alpha(\R^d),$ with $\alpha=2$ (notice that as $\alpha=2d\beta/(2\gamma +d)$ and $\beta=(d+2\gamma)(d-1)/[2(d-2\gamma)],$ $\alpha=2$ gives $\gamma=d(3-d)/4.$ Since $\gamma\geq 0$ this forces $d=3$ and so $\gamma=0.$). Thus estimate~\eqref{eq:generalization} generalizes~\eqref{eq:Hardy-type}, which corresponds to $\gamma=0$ and after letting $z$ go to zero (notice that if $\gamma=0$ the constant $C(z,V,\lambda, \mu)$ in~\eqref{eq:generalization} is no more dependent on $z,$ see~\eqref{BS-Lp}--\eqref{BS-KS}).
\end{remark}

\section{Proofs}\label{sec:proofs}
In this section we provide the proofs of
Theorems~\ref{thm:Lp-result}--\ref{thm:KS-result}
valid in dimension $d\geq 2.$ 
We give two different proofs of \Cref{thm:Lp-result}:
the proof in \Cref{sec:Lp} is strongly sensitive of the $L^p$
framework, while the proof in \Cref{sec:proofs-sub} is more robust and
it is adapted to prove also \Cref{thm:MC-result} and \Cref{thm:KS-result}.
\subsection{Proof of \Cref{thm:Lp-result}}\label{sec:Lp}
The strategy of the proof of \Cref{thm:Lp-result}  follows the
one of~\cite[Thm.~3.2]{FS17} with the modifications necessary to treat the Lam\'e
operator.

For notation convenience we define $p$ such that $p/(2-p)=\gamma+ d/2.$ Thus the assumptions on $\gamma$ give $1<p\leq 6/5$ if $d=2$ and $2d/(d+2)\leq p\leq 2(d+1)/(d+3)$ if $d\geq 3.$

Thanks to the H\"older inequality the multiplication by $V\in
L^\frac{p}{2-p}(\R^d)$ is a bounded operator from $L^{p'}(\R^d)^d$ to
$L^p(\R^d)^d$ with $1/p+1/{p'}=1.$
Let $z\in \C$ be an eigenvalue of $-\Delta^\ast + V$ in $L^2(\R^d)^d$ with eigenfunction $u.$ Since $-\Delta^\ast + V$ is defined via $m$-sectorial forms, we know a-priori that an eigenfunction satisfies $u\in H^1(\R^d)^d.$ In particular, by Sobolev embedding, $u\in L^r(\R^d)^d,$ for $2\leq r\leq 2d/(d-2)$ and so $u\in L^{p'}(\R^d)^d.$   

We start considering the easiest situation, \emph{i.e.}, when
$z\in \C\setminus [0,\infty).$
In this case the resolvent operator $(-\Delta^\ast-z)^{-1}\in
\mathcal{B}(L^p(\R^d)^d;L^{p'}(\R^d)^d )$
and from \Cref{lemma:uniform-estimate} one has 
\begin{equation}\label{eq:repeated}
	\|(-\Delta^\ast-z)^{-1}\|_{L^p(\R^d)^d\to L^{p'}(\R^d)^d}\leq N(z),
	\qquad N(z)=c_{p,d,\lambda, \mu} |z|^{-\frac{d+2}{2} + \frac{d}{p}}.
\end{equation}
 Using that $(-\Delta^\ast + V)u=zu$ one can write
\begin{equation}\label{eq:u-easy}
	u=(-\Delta^\ast -z)^{-1} (-\Delta^\ast-z)u=-(-\Delta^\ast -z)^{-1} Vu.
\end{equation}  
From the previous expression and the resolvent estimate~\eqref{eq:repeated}, one has
\begin{equation*}
	\begin{split}
		\|u\|_{L^{p'}(\R^d)^d}
		&=\|(-\Delta^\ast -z)^{-1} Vu\|_{L^{p'}(\R^d)^d} 
		\leq N(z)\|Vu\|_{L^{p}(\R^d)^d}\\
		&\leq N(z) \| V \|_{L^\frac{p}{2-p}(\R^d)} \norm{u}_{L^{p'}(\R^d)^d}.
	\end{split}
\end{equation*}  
Using that $N(z)=c_{p,d,\lambda, \mu}\abs{z}^{-\frac{d+2}{2} + \frac{d}{p}},$ we have 
\begin{equation}\label{eq:bound-thesis}
	1\leq c_{p,d,\lambda, \mu}\abs{z}^{-\frac{d+2}{2} + \frac{d}{p}} \| V \|_{L^\frac{p}{2-p}(\R^d)},
\end{equation}   
which gives the thesis once we replace $p/(2-p)=\gamma + d/2.$

It is clear that $z=0$ belongs to the right hand side
of~\eqref{eq:thesis}, then it remain to consider the case $z\in
(0,\infty).$ In this situation, since $z$ belongs to the spectrum of
the free Lamé operator $-\Delta^\ast,$ the
expression~\eqref{eq:u-easy} no longer makes sense. On the other hand,
taking $\varepsilon>0,$ the operator $(-\Delta^\ast
-z-i\varepsilon)^{-1}$ is well defined and bounded from $L^p(\R^d)^d$
to $L^{p'}(\R^d)^d.$  Thus, for $u$ such that $(-\Delta^\ast + V)u=zu,$ one considers an approximating
eigenfunction $u_\varepsilon$ defined as
\begin{equation*}
	u_\varepsilon=(-\Delta^\ast -z-i\varepsilon)^{-1} (-\Delta^\ast -z)u
	=-(-\Delta^\ast -z-i\varepsilon)^{-1}Vu.
\end{equation*}
Since $V\in \mathcal{B}(L^{p'}(\R^d)^d;L^p(\R^d)^d)$ and $(-\Delta^\ast -z-i\varepsilon)^{-1}\in \mathcal{B}(L^p(\R^d)^d;L^{p'}(\R^d)^d),$ we infer that $u_\varepsilon\in L^{p'}(\R^d)^d$ and 
\begin{equation}\label{eq:preliminary-eps}
	\begin{split}
		\|u_\varepsilon\|_{L^{p'}(\R^d)^d}&=\|(-\Delta^\ast -z-i\varepsilon)^{-1}Vu\|_{L^{p'}(\R^d)^d}\leq N(z+i\varepsilon)\|Vu\|_{L^{p}(\R^d)^d}\\
		&\leq N(z+i\varepsilon) \| V \|_{L^\frac{p}{2-p}(\R^d)} \norm{u}_{L^{p'}(\R^d)^d}.
	\end{split}
\end{equation}
From its explicit expression, one sees that $N(z+i\varepsilon)$ converges to $N(z)$ as $\varepsilon \to 0,$ thus the sequence $u_\varepsilon$ is uniformly bounded in ${L^{p'}(\R^d)^d}$ and therefore converges (up to subsequences) weakly in ${L^{p'}(\R^d)^d}.$ Now we want to show that $u_\varepsilon$ converges strongly in $L^2(\R^d)^d$ to $u$ as $\varepsilon$ approaches zero. Due to the $L^2$ orthogonality  of the $S$ and $P$ component of the Helmholtz decomposition it is enough to check that $(u_\varepsilon)_S$ converges to $u_S,$ the convergence of $(u_\varepsilon)_P$ to $u_P$ follows similarly.  
Using the expressions~\eqref{eq:simple_writing} and~\eqref{eq:resolvent} and applying Plancherel theorem one has
\begin{equation*}
	\begin{split}
		\|(u_\varepsilon)_S-u_S\|_{L^2(\R^d)^d}
		&=\big\|\big [\big (-\Delta - \tfrac{z+i\varepsilon}{\mu} \big)^{-1} (-\Delta - \tfrac{z}{\mu})- I\big ]u_S\big \|_{L^2(\R^d)^d}\\
		&=\big\|\big [\big (|\xi|^2- \tfrac{z+i\varepsilon}{\mu} \big)^{-1} (|\xi|^2 - \tfrac{z}{\mu})- 1\big ]\widehat{u}_S\big \|_{L^2(\R^d)^d},
	\end{split}
\end{equation*} 
then the conclusion follows from dominated convergence theorem.

To show that $u_\varepsilon$ converges weakly to $u$ in $L^{p'}$, it is enough to prove that $\langle u_\varepsilon, \varphi \rangle_{p' p}$ converges to $\langle u, \varphi \rangle_{p' p}$ for all $\varphi \in L^2\cap L^p,$ that is immediate from Cauchy-Schwarz inequality and the strong convergence of $u_\varepsilon$ to $u$ in $L^2.$

Finally, using the weak lower semi-continuity of the norm and the preliminary estimate~\eqref{eq:preliminary-eps}, one has
\begin{equation*}
	\begin{split}
		\|u\|_{L^{p'}(\R^d)^d}
		&\leq \liminf_{\varepsilon \to 0} \| u_\varepsilon\|_{L^{p'}(\R^d)^d}
		\leq \liminf_{\varepsilon \to 0} N(z+i\varepsilon) \| V \|_{L^\frac{p}{2-p}(\R^d)} \norm{u}_{L^{p'}(\R^d)^d}\\
		&=N(z) \| V \|_{L^\frac{p}{2-p}(\R^d)} \norm{u}_{L^{p'}(\R^d)^d}.
	\end{split}
\end{equation*} 
From this, as above, one concludes that the bound~\eqref{eq:bound-thesis} holds, which gives the thesis. 
\qed

\begin{remark}\label{rmk:two-steps}
  In the proof of the previous result two
        ingredients have been used in a crucial way: the uniform resolvent estimate~\eqref{Lame_res_1} from Lemma~\ref{lemma:uniform-estimate}, which holds true for spectral parameters outside the spectrum $\sigma(-\Delta^\ast)=[0,\infty)$, and the continuity of $N(z)$ up to $(0,\infty),$ that is $N(z+i\varepsilon) \to N(z)$ as $\varepsilon$ goes to zero and $z\in (0,\infty),$ which allows us to cover also the case of possible embedded eigenvalues $z\in (0,\infty).$
\end{remark}
	
	Recently, Kwon, Lee and Seo~\cite{KLS20}, adapting recent sharp resolvent estimates obtained for the Laplacian by two of the three authors in~\cite{KL20}, were able to prove analogous estimates for the Lamé operator, which improve the one stated in Lemma~\ref{lemma:uniform-estimate}. More precisely, they proved the following result:
	\begin{theorem}[{\cite[Theorem 1.3]{KLS20}}]
	\label{thm:KLS}
		Let $d\geq 2,$ $z\in \C\setminus [0,\infty),$ $1<p\leq q<\infty.$ If $(\frac{1}{p}, \frac{1}{q})\in \mathcal{R}_1 \cup \widetilde{\mathcal{R}}_2 \cup \widetilde{\mathcal{R}}_3 \cup \widetilde{\mathcal{R}}_3'$ (see ~\cite[Def.~1.1]{KLS20}), then one has
		\begin{equation*}
			\|(-\Delta^\ast -z)^{-1}\|_{L^p(\R^d)\to L^q(\R^d)}\leq N(z),
		\end{equation*}
		where $N(z)=c_{p,q,d, \lambda, \mu}|z|^{-1 + \frac{d}{2} \big( \frac{1}{p}-\frac{1}{q}\big)} \dist(z/|z|, [0,\infty))^{-\gamma_{p,q}},$ with
		$\gamma_{p,q}:= \max \{0, 1- \frac{d+1}{2}(\frac{1}{p}-\frac{1}{q}), \frac{d+1}{2}-\frac{d}{p}, \frac{d}{q}- \frac{d-1}{2}\}.$
	\end{theorem}
	As a consequence, the following result on location of discrete eigenvalues is also proven.
	\begin{corollary}[{\cite[Corollary 1.4]{KLS20}}]
	\label{cor:KLS}
		Let $d\geq 2,$ $1<p\leq q<\infty,$ $(\frac{1}{p}, \frac{1}{q})\in \mathcal{R}_1 \cup \widetilde{\mathcal{R}}_2 \cup \widetilde{\mathcal{R}}_3 \cup \widetilde{\mathcal{R}}_3'$ (see ~\cite[Def.~1.1]{KLS20}). Then any eigenvalue $z\in \C\setminus [0,\infty)$ of $-\Delta^\ast + V$ acting on $L^q(\R^d)^d$ satisfies
		\begin{equation*}
			|z|^{1-\frac{d}{2}(\frac{1}{p}-\frac{1}{q})} \dist(z/|z|, [0,\infty))^{\gamma_{p,q}}\leq c_{p,q,d,\lambda, \mu}\|V\|_{L^\frac{pq}{q-p}(\R^d)}.
		\end{equation*}
	\end{corollary}
	For the explicit expressions of the ranges $\mathcal{R}_1,
        \widetilde{\mathcal{R}}_2, \widetilde{\mathcal{R}}_3$ and
        $\widetilde{\mathcal{R}}_3'$ of allowed indexes $p,q$ we refer
        the reader to the original paper~\cite{KLS20} (Definition 1.1
        and Figure 1 and 2 there); we give a few comments on their
        result here. The region $\mathcal{R}_1$ is represented by $p, q$ such that
	\begin{equation}\label{eq:Gutierrez}
		\frac{2}{d+1}\leq \frac{1}{p}-\frac{1}{q}\leq \frac{2}{d}, 
		\qquad \frac{1}{p}>\frac{d+1}{2d},
		\qquad \frac{1}{q}<\frac{d-1}{2d},
	\end{equation}
	and $\frac{1}{p}-\frac{1}{q}\neq \frac{2}{d}$ if $d=2.$
	In particular, the duality line $\frac{1}{p} + \frac{1}{p'}=1$
        restricted to $\frac{2d}{d+2}\leq p\leq \frac{2(d+1)}{d+3}$ of
        estimate~\eqref{Lame_res_1} is contained in the range
        $\mathcal{R}_1.$ Notice that if $p,q$
        satisfies~\eqref{eq:Gutierrez}, then $\gamma_{p,q}$ in
        \Cref{thm:KLS} is zero, that is, within $\mathcal{R}_1,$ $N(z)$ depends only on $|z|,$ whereas outside $\mathcal{R}_1$ it also depends on the distance from the spectrum $\sigma(-\Delta^\ast)=[0,\infty).$ Thus, in light of Remark~\ref{rmk:two-steps} above, outside $\mathcal{R}_1,$ since $N(z)$ becomes singular as $z$ approaches the positive real axis, Corollary~\ref{cor:KLS} cannot be improved to cover also possible embedded eigenvalues $z\in (0,\infty).$ 
	Finally, notice that the range $\mathcal{R}_1$ allows for a larger collection of indexes $p,q$ than just the self-dual case $p, p'.$ On the other hand, as long as one is interested in finding bounds on the location of eigenvalues in terms of norms of the potential, considering the whole range $\mathcal{R}_1$ (instead of just the duality line $1/p+1/p'=1$) does no provide with more information. Indeed, in this context, it is the local integrability/asymptotic behavior of the potential $V$ that matters, or better for which class of potential such a bound holds true, in other words one takes into account not the pair $(1/p, 1/q)$ but rather the difference $1/r:=1/p-1/q$ (notice that $pq/(q-p)=r$ and $V\in L^r$).

\subsection{Proof of
  \Cref{thm:Lp-result}, \Cref{thm:MC-result} and \Cref{thm:KS-result}}\label{sec:proofs-sub}
  In this section we use an adaptation of the Birman-Schwinger principle
  in the spirit of the proof provided above.
  Nonetheless, the proof presented here  allows to treat at
  a time the $L^p$ framework as well as the Morrey-Campanato and the Kerman-Sawyer setting.

Let $z\in \C$ be an eigenvalue of $-\Delta^\ast + V$ in $L^2(\R^d)^d$
with corresponding eigenfunction $u\in H^1(\R^d)^d.$ We first consider
the easiest case of eigenvalues outside the spectrum of
$-\Delta^\ast,$ namely $z\in \C\setminus [0,\infty).$ In this
situation the standard Birman-Schwinger principle applies: if $z\in
\C\setminus [0,\infty)$ is an eigenvalue of the perturbed Lamé
operator $-\Delta^\ast+V$ with corresponding eigenfunction $u\in H^1(\R^d)^d,$ then
$-1$ is an eigenvalue of the Birman-Schwinger operator
$K_z:=|V|^{1/2}(-\Delta^\ast-z)^{-1}V_{1/2}$ with eigenvector
$\phi:=|V|^{1/2}u \in L^2(\R^d)^d$
(see, for example, Thm. III.12 and Thm. III.14 in~\cite{Simon71}).
Notice that $\phi:=|V|^{1/2}u\in L^2(\R^d)^d$ for any $u\in
H^1(\R^d)^d$ by Lemma~\ref{lemma:Frank}.
Since	$\phi=-|V|^{1/2}(-\Delta^\ast -z)^{-1} V_{1/2}\phi,$
using the bounds in~\Cref{lemma_BS-estimates} we get 
	\begin{equation*}
		\begin{split}
			\|\phi\|_{L^2(\R^d)^d}
			&\leq \||V|^{1/2}(-\Delta^\ast -z)^{-1} V_{1/2}\|_{L^2\to L^2}\|\phi\|_{L^2(\R^d)^d}\\
			&\leq c_{\gamma,d, \lambda, \mu} |z|^{-\frac{2\gamma}{2\gamma + d}} \|V\| \|\phi\|_{L^2(\R^d)^d},
		\end{split}
	\end{equation*}
	where $\|V\|$ denotes $\|V\|_{L^{\gamma + \frac{d}{2}}(\R^d)},$ $\|V\|_{\mathcal{L}^{\alpha, p}(\R^d)}$ or $Q_2(|V|)^2\, \||V|^\beta\|_{\mathcal{KS}_\alpha(\R^d)}^\frac{1}{\beta}$ depending on which operator estimate from Lemma~\ref{lemma_BS-estimates} we used to bound the norm of the Birman-Schwinger operator, namely~\eqref{BS-Lp},~\eqref{BS-MC} or~\eqref{BS-KS}, respectively. This gives the proof of Theorems~\ref{thm:Lp-result}--\ref{thm:KS-result} for $z\in \C\setminus [0,\infty).$

Now, let $z\in [0,\infty).$ Observe that for any $\varepsilon>0$ and $z\in[0,\infty)$ the operator  $(-\Delta^\ast-z-i\varepsilon)^{-1}$ is well defined. The approximating eigenfunction
$u_\varepsilon:=(-\Delta^\ast-z-i\varepsilon)^{-1}(-\Delta^\ast-z)u,$
 satisfies the corresponding problem 
\begin{equation*}
	(-\Delta^\ast -z-i\varepsilon)u_\varepsilon =-Vu.
\end{equation*} 
Defining the auxiliary functions $\phi:=|V|^{1/2}u$ and $\phi_\varepsilon:=|V|^{1/2}u_\varepsilon$ one easily gets the following identity
\begin{equation*}
  \phi_\varepsilon= - |V|^{1/2}(-\Delta^\ast -z-i\varepsilon)^{-1} V_{1/2}\phi.
\end{equation*}
Passing to the norms and using~\Cref{lemma_BS-estimates} we get
\begin{equation*}
	\begin{split}
		\|\phi_\varepsilon\|_{L^2(\R^d)^d}
		&\leq \||V|^{1/2}(-\Delta^\ast -z-i\varepsilon)^{-1} V_{1/2}\|_{L^2\to L^2}\|\phi\|_{L^2(\R^d)^d}\\
		&\leq c_{\gamma,d, \lambda, \mu} (|z|^2 + \varepsilon^2)^{-\frac{\gamma}{2\gamma + d}}
	\|V\| \|\phi\|_{L^2(\R^d)^d},
	\end{split}
\end{equation*}
where, as above, $\|V\|$ denotes $\|V\|_{L^{\gamma + \frac{d}{2}}(\R^d)},$ $\|V\|_{\mathcal{L}^{\alpha, p}(\R^d)}$ or $Q_2(|V|)^2\, \||V|^\beta\|_{\mathcal{KS}_\alpha(\R^d)}^\frac{1}{\beta}.$
Thus the theses of \Cref{thm:Lp-result},
\Cref{thm:MC-result} and \Cref{thm:KS-result} follow
letting $\varepsilon$ go to zero as soon as one proves that
$\phi_\varepsilon$ converges to $\phi$ in $L^2(\R^d)^d.$ Notice first
that using the dominated convergence theorem in Fourier space, one
easily checks as in Subsection~\ref{sec:Lp} that
$u_\varepsilon:=(-\Delta^\ast-z-i\varepsilon)^{-1} (-\Delta^\ast-z)u$
converges to $u$ in $H^1(\R^d)^d.$ Then the convergence of
$\phi_\varepsilon$ to $\phi$ in $L^2(\R^d)^d$ follows as a consequence
the boundedness of $|V|^{1/2}$ as an operator from $H^1(\R^d)^d$ to
$L^2(\R^d)^d$, see Lemma~\ref{lemma:Frank}. 
\qed

\section{Spectral stability in three dimensions}\label{sec:proof-d3}
This section is devoted to the proof of Theorems~\ref{thm:dim3}--\ref{thm:L^p-d3} which show that the spectrum of the perturbed Lamé operator in $d=3$ remains stable under suitable small perturbations (\emph{cfr.}~\eqref{eq:cond-FKV-d3},~\eqref{eq:cond-d3-M-C} and~\eqref{eq:cond-d3-Lp}).

We first prove~\Cref{thm:dim3}:~\Cref{thm:M-Cd3} and~\Cref{thm:L^p-d3}
are obtained with minor modifications of the argument.
The proof of \Cref{thm:dim3} 
follows the strategy developed in~\cite{FKV18} to prove the analogous
result for three dimensional Schr\"odinger operators
and it will be obtained as a consequence of some preliminary results which are contained in the following subsections. The final proof of Theorems~\ref{thm:dim3}, and then of~\Cref{thm:M-Cd3} and~\Cref{thm:L^p-d3}, can be found in~\Cref{sub:dim=3}.

In the following lemma we show that under the
assumption~\eqref{eq:cond-FKV-d3}
the Birman-Schwinger operator $K_z$ is bounded with bound strictly
less than one, using the explicit formula~\eqref{eq:integral-kernel} for the Green function $\mathcal{G}_z(x,y)$ of $-\Delta-z$.
\begin{lemma}\label{lemma:bdd-K_z}
	Let $d=3$ and assume~\eqref{eq:cond-FKV-d3}. Then there exists a positive constant $\mathfrak{a}<1$ such that
	\begin{equation}\label{eq:K_z-bound-frak}
          \|K_z\|_{L^2(\R^3)^3 \to L^2(\R^3)^3}\leq \mathfrak{a},
		\qquad \text{for all}\; z\in \C\setminus (0,\infty).
	\end{equation}
\end{lemma}
\begin{proof}
	To bound the operator norm of $K_z$ we estimate the inner
        product $\langle f, K_z g \rangle,$ for any $f,g \in
        L^2(\R^3)^3.$ Using the same strategy of the proof of \Cref{lemma_BS-estimates-d=3} one has 
	\begin{equation*}
		|\langle f, K_z g \rangle|\leq a \frac{1+ 6c_V^2}{\min\{\mu, \lambda + 2\mu\}} \|f\|_{L^2(\R^3)^3} \|g\|_{L^2(\R^3)^3}.
	\end{equation*}
	Thanks to~\eqref{eq:cond-FKV-d3}, we get the thesis
        for $\mathfrak{a}:=a (1+ 6c_V^2)/\min\{\mu, \lambda +
        2\mu\} <1$.
\end{proof}

\subsection{Absence of eigenvalues}  
As a starting point we observe that under the assumption~\eqref{eq:cond-FKV-d3}, Corollary~\ref{cor:KS} ensure the absence of the point spectrum, more precisely we have the following result. 

\begin{proposition}[Absence of eigenvalues]\label{thm:absence-evs-d3}
	Let $d=3$ and assume~\eqref{eq:cond-FKV-d3}. Then $\sigma_\textup{p}(-\Delta^\ast+V)=\varnothing.$
	\begin{proof}
		Taking into account~\Cref{rmk:d=3-KS} the proposition is an easy consequence of Corollary~\ref{cor:KS}. 
	\end{proof}
\end{proposition}

The next step we accomplish is to show the absence of the continuous spectrum outside $[0,\infty).$
	\subsection{Absence of the continuous spectrum outside \texorpdfstring{$[0,\infty)$}{infty}}
	We need the following lemma which is valid for any dimension $d\geq 3.$
	\begin{lemma}\label{lemma:singular-sequences}		
	Let $d\geq 3$ and assume~\eqref{eq:cond-FKV-d3}. If $\|(-\Delta^\ast + V)u_n -zu_n\|_{L^2(\R^d)^d}\to 0$ as $n\to \infty$ with some $z\in \C \setminus \R$ and $\{u_n\}_{n\in \N}\subset H^1(\R^d)^d$ such that $\|u_n\|_{L^2(\R^d)^d}=1$ for all $n\in \N,$ then $\phi_n:=|V|^{1/2}u_n$ obeys
\begin{equation*}
  \lim_{n\to \infty} \frac{\langle \phi_n, K_z \phi_n \rangle}{\|\phi_n\|_{L^2(\R^d)^d}^2}=-1.
\end{equation*}
\begin{proof}
Given $z\in \C\setminus \R$ and using the explicit representation of the resolvent given in~\eqref{eq:resolvent} one has 
		\begin{equation}\label{eq:I+II}
			\begin{split}
				\langle \varphi_n, K_{z} \phi_n \rangle
				&=\langle \varphi_n, |V|^{1/2}(-\Delta^\ast -z)^{-1}V_{1/2}\phi_n \rangle
				=\langle \varphi_n, |V|^{1/2}(-\Delta^\ast -z)^{-1}Vu_n \rangle\\
				&=\frac{1}{\mu} \langle \varphi_n, |V|^{1/2}(-\Delta -\tfrac{z}{\mu})^{-1}(Vu_n)_\textup{S} \rangle
				+\frac{1}{\lambda + 2\mu} \langle \varphi_n, |V|^{1/2}(-\Delta -\tfrac{z}{\lambda+ 2\mu})^{-1}(Vu_n)_\textup{P}\rangle\\
				&=I+II.
			\end{split}
		\end{equation}
	We consider only $I$ as $II$ can be treated similarly.

	Defining $F_n= F_{n,\textup{S}}+ F_{n,\textup{P}}:=Vu_n$ we have
	\begin{equation}\label{eq:I}
		\begin{split}
			I&:=\frac{1}{\mu}\langle \varphi_n, |V|^{1/2}(-\Delta -\tfrac{z}{\mu})^{-1}F_{n,\textup{S}}\rangle\\
			&=\frac{1}{\mu}\iint_{\R^3\times \R^3} \overline{\varphi_n(x)}|V|^{1/2}(x) \mathcal{G}_{\tfrac{z}{\mu}}(x,y) F_{n,\textup{S}}(y)\, dx\, dy\\
			&=\frac{1}{\mu} \int_{\R^3} \eta_{n, \mu}(y)F_{n,\textup{S}}(y)\, dy,
		\end{split}
	\end{equation}
	where 
	\begin{equation}\label{eq:eta}
		\eta_{n,\mu}:=\int_{\R^3} \mathcal{G}_{\tfrac{z}{\mu}}(x,\cdot)|V|^{1/2}(x)\overline{\varphi_n(x)}\, dx=(-\Delta - \tfrac{z}{\mu})^{-1}|V|^{1/2}\overline{\varphi_n},
	\end{equation}
	where the second equality holds due to the symmetry $\mathcal{G}_\zeta(x,y)=\mathcal{G}_\zeta(y,x).$
	
	The analogous computations for $II$ give
	\begin{equation}\label{eq:II}
		II=\frac{1}{\lambda + 2\mu} \int_{\R^3} \eta_{n, \lambda + 2\mu}(y)F_{n,\textup{P}}(y)\, dy,
	\end{equation}
	where $\eta_{n, \lambda + 2\mu}$ is defined analogously to $\eta_{n,\mu}$ in~\eqref{eq:eta}.
	
	Using~\eqref{eq:I} and~\eqref{eq:II} in~\eqref{eq:I+II} gives 
  \begin{equation}\label{eq:integral-n}
    \langle \phi_n, K_z \phi_n \rangle
    =\frac{1}{\mu} \int_{\R^d} \eta_{n,
      \mu}(y) F_{n,\textup{S}}(y)\,dy
    + \frac{1}{\lambda + 2 \mu} \int_{\R^d} \eta_{n, \lambda + 2\mu}(y) F_{n,\textup{P}}(y)\,dy.
  \end{equation}

Notice that
  $\eta_{n,\mu},\eta_{n,\lambda + 2\mu} \in  H^1(\R^d)^d.$
  Indeed, writing $H_0:=-\Delta,$
  we have
			\begin{equation}\label{eq:eta_n_mu}
                          \eta_{n, \mu}=(H_0- z/\mu)^{-1} H_0^{1/2} H_0^{-1/2} |V|^{1/2} \overline{\phi_n},
			\end{equation}				
			since $\phi_n\in L^2(\R^d)^d$ by~\eqref{eq:cond-FKV-d3}, $H_0^{-1/2}|V|^{1/2}$ is bounded due to~\eqref{equiv} and $(H_0- z/\mu)^{-1} H_0^{1/2}$ maps $L^2(\R^d)^d$ to $H^1(\R^d)^d,$ one has 
			\begin{equation}\label{eq:eta-norm}
				\|\eta_{n,\mu}\|_{L^2(\R^d)^d}\leq C_{z/\mu} \sqrt{a} \|\phi_n\|_{L^2(\R^d)^d},
				\quad \text{where} \quad
				C_{z/\mu}:=\sup_{\xi\in [0,\infty)} \left| \frac{\xi}{\xi^2-z/\mu}\right|.
			\end{equation}
			Due to the $L^2$-orthogonality of the $S$ and
                        $P$ component of the Helmholtz decomposition and using that the projection into the $S$ and
                        $P$ components commutes with the Laplacian (\emph{cfr.}~\eqref{eq:Helmholtz-Riesz}),
                        from $\|(-\Delta^\ast + V) u_n
                        -zu_n\|_{L^2(\R^d)^d}\to 0,$ we get
			\begin{equation}\label{eq:convergence}
				\|(-\Delta -z/\mu) (u_n)_\textup{S} +\tfrac{1}{\mu} F_{n,\textup{S}}\|_{L^2(\R^d)^d}\to 0,
				\quad \text{and} \quad
				\|(-\Delta -z/(\lambda + 2\mu)) (u_n)_\textup{P} +\tfrac{1}{\lambda + 2\mu} F_{n,\textup{P}}\|_{L^2(\R^d)^d}\to 0.
			\end{equation}
			Let us define the following quantities
			\begin{equation*}
				R_\mu:=\langle \nabla \overline{\eta_{n, \mu}}, \nabla (u_n)_\textup{S}\rangle - \frac{z}{\mu} \langle \overline{\eta_{n, \mu}}, (u_n)_\textup{S} \rangle,
				\qquad
				R_{\lambda + 2\mu}:=\langle \nabla \overline{\eta_{n, \lambda+ 2\mu}}, \nabla (u_n)_\textup{P}\rangle - \frac{z}{\lambda + 2\mu} \langle \overline{\eta_{n, \lambda + 2\mu}}, (u_n)_\textup{P} \rangle.
			\end{equation*}
			Thanks to~\eqref{eq:eta_n_mu}, we have
			\begin{equation*}
				\begin{split}
					R_\mu
					&=\langle \nabla \overline{ (u_n)_\textup{S}}, \nabla \eta_{n, \mu}\rangle - \frac{z}{\mu} \langle \overline{(u_n)_\textup{S}}, \eta_{n, \mu} \rangle\\
					&=\langle H_0^{1/2} \overline{ (u_n)_\textup{S}}, H_0^{1/2} (H_0-z/\mu)^{-1} H_0^{1/2} H_0^{-1/2} |V|^{1/2}\overline{\phi_n}\rangle - \frac{z}{\mu} \langle \overline{(u_n)_\textup{S}}, \eta_{n, \mu} \rangle\\
					&=\langle H_0^{1/2} \overline{(u_n)_\textup{S}}, H_0^{-1/2} |V|^{1/2} \overline{\phi_n}\rangle\\
					&\phantom{=}+\frac{z}{\mu} \langle  H_0^{1/2} \overline{(u_n)_\textup{S}}, (H_0-z/\mu)^{-1} H_0^{-1/2}|V|^{1/2} \overline{\phi_n}\rangle
					  - \frac{z}{\mu} \langle \overline{(u_n)_\textup{S}}, \eta_{n, \mu} \rangle\\
					  &=\langle H_0^{1/2} \overline{(u_n)_\textup{S}}, H_0^{-1/2} |V|^{1/2} \overline{\phi_n}\rangle\\
					  &=\langle (H_0^{-1/2} |V|^{1/2})^\ast H_0^{1/2} \overline{(u_n)_\textup{S}},  \overline{\phi_n}\rangle\\
					  &=\langle |V|^{1/2} \overline{(u_n)_\textup{S}}, \overline{\phi_n}\rangle. 
				\end{split}
			\end{equation*}
			Similar computations for $R_{\lambda + 2\mu}$ give
			\begin{equation*}
				R_{\lambda + 2\mu}
				=\langle |V|^{1/2} \overline{(u_n)_\textup{P}}, \overline{\phi_n}\rangle. 
			\end{equation*}
			Adding and subtracting the quantities $R_\mu$ and $R_{\lambda + 2\mu}$ to~\eqref{eq:integral-n} and noticing that 
			\begin{equation*}
				R_\mu + R_{\lambda + 2\mu}=\langle |V|^{1/2} \overline{(u_n)_\textup{S}}, \overline{\phi_n}\rangle+\langle |V|^{1/2} \overline{(u_n)_\textup{P}}, \overline{\phi_n}\rangle=\langle |V|^{1/2} \overline{(u_n)}, \overline{\phi_n}\rangle=\| \phi_n\|_{L^2(\R^d)^d}^2,
				\end{equation*} 
			one has
			\begin{equation}\label{eq:before}
					\langle \phi_n, K_z \phi_n \rangle
				= R_\mu + \frac{1}{\mu} \int_{\R^d} \eta_{n, \mu}(y) F_{n,\textup{S}}(y)\,dy
				+R_{\lambda+2\mu} + \frac{1}{\lambda + 2 \mu} \int_{\R^d} \eta_{n, \lambda + 2\mu}(y) F_{n,\textup{P}}(y)\,dy
				-\| \phi_n\|_{L^2(\R^d)^d}^2				
			\end{equation}
Since
\begin{equation*}
  \begin{split}
    \|(-\Delta^\ast + V)u_n -zu_n\|_{L^2(\R^d)^d}
    &=\sup_{\substack{\varphi\in L^2(\R^d)^d\\ \varphi \neq 0}} 
    \frac{\langle \varphi, (-\Delta^\ast + V)u_n -zu_n \rangle}{\|\varphi\|_{L^2(\R^d)^d}}\\
    &\geq |\mu \|\nabla u_{n,\textup{S}}\|_{L^2(\R^d)^d}^2 + (\lambda + 2\mu)\|\nabla u_{n,\textup{P}}\|_{L^2(\R^d)^d} + v[u_n]-z|,
  \end{split}
\end{equation*}	
	where the inequality is obtained choosing $\varphi=u_n,$ and the left-hand side vanishes as $n$ goes to infinity, we have $\Im v[u_n]$ tends to $\Im z\neq 0$ as $n$ goes to infinity. In particular, $\liminf_{n\to \infty} \|\phi_n\|_{L^2(\R^d)^d}>0.$			
	
From~\eqref{eq:before} one has
			\begin{equation*}
				\begin{split}
				\frac{\langle \phi_n, K_z \phi_n \rangle}{\|\phi_n\|_{L^2(\R^d)^d}^2}
				&=\frac{1}{\|\phi_n\|_{L^2(\R^d)^d}^2} \Big[ R_\mu + \frac{1}{\mu} \int_{\R^d} \eta_{n, \mu}(y) F_{n,\textup{S}}(y)\,dy \Big]\\
				&\phantom{=}+\frac{1}{\|\phi_n\|_{L^2(\R^d)^d}^2} \Big[ R_{\lambda+2\mu} + \frac{1}{\lambda + 2 \mu} \int_{\R^d} \eta_{n, \lambda + 2\mu}(y) F_{n,\textup{P}}(y)\,dy\Big]
				-1	\\
				&=I+II-1.
				\end{split}
			\end{equation*}
		Now we show that $I$ and $II$ tend to zero as $n$ goes to infinity. Using the explicit expressions for $R_\mu$ and  $R_{\lambda +2\mu}$ and estimate~\eqref{eq:eta-norm}, one has
		\begin{equation*}
			\begin{split}
				|I|&=\frac{|\langle\overline{\eta_{n,\mu}}, (-\Delta -z/\mu) (u_n)_\textup{S} + \tfrac{1}{\mu} F_{n,\textup{S}}\rangle|}{\|\phi_n\|_{L^2(\R^d)^d}^2}
				\leq \frac{\|\eta_{n,\mu}\|_{L^2(\R^d)^d}\|(-\Delta -z/\mu) (u_n)_\textup{S} + \tfrac{1}{\mu} F_{n,\textup{S}} \|_{L^2(\R^d)^d}}{\|\phi_n\|_{L^2(\R^d)^d}^2}\\
				&\leq C_{z/\mu} \sqrt{a} \frac{\|(-\Delta -z/\mu) (u_n)_\textup{S} + \tfrac{1}{\mu} F_{n,\textup{S}} \|_{L^2(\R^d)^d}}{\|\phi_n\|_{L^2(\R^d)^d}}. 
			\end{split}
		\end{equation*} 
		Since $\liminf_{n\to
                  \infty}\|\phi_n\|_{L^2(\R^d)^d}>0$ and using~\eqref{eq:convergence} we conclude that the right hand side tends to zero as $n$ goes to infinity. Analogous computations show that also $II$ vanishes as $n\to \infty.$ This yields
		\begin{equation*}
			\lim_{n\to \infty} \frac{\langle \phi_n, K_z \phi_n \rangle}{\|\phi_n\|_{L^2(\R^d)^d}^2}=-1
		\end{equation*}
			and then the proof is concluded.
		\end{proof}
	\end{lemma}
Now we are in position to prove that there is no continuous spectrum outside $[0,\infty).$
\begin{proposition}\label{thm:absence-continuous}
	Let $d=3$ and assume~\eqref{eq:cond-FKV-d3}. Then $\sigma_\textup{c}(-\Delta^\ast+V)\subset [0,\infty).$
	\begin{proof}
          Consider $\Re h_V[u],$ where $h_V[u]$ is the quadratic form associated with $-\Delta^\ast + V$ (see~\eqref{eq:h_0},\eqref{eq:v}). One has
	\begin{equation*}
		\Re h_V[u]=\mu \int_{\R^3} |\nabla u_\textup{S}|^2\, dx + (\lambda +2\mu) \int_{\R^3} |\nabla u_\textup{P}|^2\, dx 
		+ \Re \int_{\R^3} \overline{Vu}\cdot u\, dx.
              \end{equation*}
              By assumption~\eqref{eq:cond-FKV-d3}, $\Re h_V[u]\geq (\min\{	\mu, \lambda +2\mu\}-a)\|\nabla u\|^2\geq 0$ for all $u\in H^1(\R^3)^3.$ Since $-\Delta^\ast + V$ is m-sectorial, then its spectrum is contained in the right complex half-plane (\emph{cf}.~\cite[Thm. V.3.2]{Kato}). Now, assume by contradiction that there exists $z\in \C$ with $\Re z\geq 0$ and $\Im z \neq 0$ such that $z\in \sigma_\textup{c}(-\Delta^\ast + V).$ Then $z$ belongs to the kind of essential spectrum which is characterized by the existence of a singular sequence of $-\Delta^\ast + V$ corresponding to $z$ (\emph{cf.}~\cite[Thm. IX.1.3]{Ed_Ev}): there exists $\{u_n\}_{n\in \N}\subset H^1(\R^3)^3$ such that $\| u_n\|_{L^2(\R^3)^3}=1$ for all $n\in \N,$ $\|(-\Delta^\ast + V-z)u_n\|_{L^2(\R^3)^3}\to 0$ as $n\to \infty$ and $\{u_n\}_{n\in \N}$ is weakly converging to zero. By Lemma~\ref{lemma:singular-sequences} and~\eqref{eq:K_z-bound-frak}, one has
		  \begin{equation*}
		  	\mathfrak{a}>\|K_z\|\geq \Big |\lim_{n\to \infty} \frac{\langle u_n, K_z u_n\rangle}{\|u_n\|_{L^2(\R^3)^3}^2} \Big|=1,
		  \end{equation*}   
		  which is a contradiction as $\mathfrak{a}<1.$
	\end{proof}
\end{proposition}

\subsection{Inclusion of \texorpdfstring{$[0,\infty)$}{zeroinfty} in the spectrum}
Now we show that the semi axis $[0,\infty)$ lies in the spectrum. In order to do that we shall use the following criterion.
\begin{lemma}[{\cite[Lemma 4]{FKV18}}]\label{lemma:FKV}
	Let $H$ be an m-sectorial accretive operator in a complex Hilbert space $\mathcal{H}$ which is associated with a densely defined, closed, sectorial) sesquilinear form $h.$ Given $z \in \C,$ assume that there exists a sequence $\{\phi_n\}_{n\in \N}\subset \mathcal{D}(h)$ such that $\|\phi_n\|_{\mathcal{H}}=1$ for all $n\in \N$ and
	\begin{equation}\label{eq:sequence-spectr}
		\sup_{\substack{\psi\in \mathcal{D}(h)\\ \psi \neq 0}} 
		\frac{|h(\phi_n, \psi)-z(\phi_n,\psi)|}{\|\psi\|_{\mathcal{D}(h)}}
		\xrightarrow[n\to \infty]{}
		0,
	\end{equation}
	where $\|\psi\|_{\mathcal{D}(h)}:=\sqrt{\Re h[\psi] +\|\psi\|^2}.$ Then $z\in \sigma(H).$
\end{lemma}

In the following we 
construct an appropriate sequence $\{\phi_n\}_{n\in \N}$ to apply Lemma~\ref{lemma:FKV} to $H=-\Delta^\ast + V$ and $z\in [0,\infty),$ 
 showing then that
$[0,\infty)\subset \sigma(-\Delta^\ast + V).$ In~\cite{FKV18} the
authors proved the analogous result for the Schr\"odinger operator
taking as $\{\phi_n\}_{n\in \N}$ the standard singular sequence for
the Laplacian.
In order to adapt that construction to this setting, we perform a suitable diagonalization argument operated on the symbol of the Lamé operator.

\begin{lemma}
\label{lemma:diagonalization}
Let $d\geq 3.$ For any $z\in (0,\infty)$ there exists a
classical solution $u\in C^\infty(\R^d)^d$
to
\begin{equation}\label{eq:free-evl}
  -\Delta^\ast u - zu=0
\end{equation}	 
such that $| u (x)| =1$ for all $x\in \R^3$ and its
derivatives are bounded.
\end{lemma}
\begin{proof}
For simplicity of notation, we give a proof in the case that $d=3$. The general case $d\geq 3$
is adapted straightforwardly.

From the explicit form of the Lamé operator $-\Delta^\ast$,
$u$ is solution
to \eqref{eq:free-evl} if and only if its Fourier trasform
 $\widehat{u} := \mathcal{F} u$ satisfies
\begin{equation*}
  L(\xi)\widehat{u}(\xi)-z\widehat{u}(\xi)=0, \quad \text{ for
    a.a. }\xi \in \R^3,
\end{equation*}
where
\begin{equation*}
  L(\xi)=\mu |\xi|^2 \widehat{u}(\xi) + (\lambda + \mu) \xi \xi^t
  \widehat{u}(\xi)
  =
  \begin{pmatrix}
    \mu|\xi|^2 + (\lambda +\mu)\xi_1^2 & (\lambda + \mu) \xi_1\xi_2 &(\lambda + \mu) \xi_1\xi_3\\
    (\lambda + \mu) \xi_1\xi_2 & \mu|\xi|^2 + (\lambda +\mu)\xi_2^2 & (\lambda + \mu) \xi_2\xi_3\\
    (\lambda + \mu) \xi_1\xi_3 & (\lambda + \mu) \xi_2\xi_3 & \mu|\xi|^2 + (\lambda +\mu)\xi_3^2 
  \end{pmatrix}.
\end{equation*}
For a.e.~$\xi \in \R^3$ we have that 
\begin{equation*}
  P^{-1}(\xi)\,L(\xi) \, P(\xi)  =   D(\xi):=
  \begin{pmatrix}
    \mu |\xi|^2 & 0 & 0 \\
    0 & \mu |\xi|^2 & 0 \\
    0 & 0 & (\lambda +2\mu)|\xi|^2
  \end{pmatrix}
  ,
  \quad \text{ with }
  P(\xi)=
  \begin{pmatrix}
    -\xi_2 & -\xi_3 &\xi_1\\
    \xi_1 & 0 & \xi_2\\
    0 & \xi_1 &\xi_3
  \end{pmatrix}
  .
\end{equation*}
 Determining a solution $u$ of~\eqref{eq:free-evl} is equivalent to find 
 a vector field
 $\widehat{v}=(\widehat{v}_1,\widehat{v}_2,\widehat{v}_3):=P^{-1}\widehat{u}$
 such that $D(\xi)\widehat{v}(\xi)-z\widehat{v}(\xi)=0.$ 
Using the inverse Fourier transform, one is reduced to determine $v=(v_1, v_2, v_3)$ a solution to the Helmholtz-type system
\begin{equation}\label{eq:Helmholtz-type}
	\begin{system}
		-\Delta v_1-\frac{z}{\mu} v_1=0,\\
		-\Delta v_2-\frac{z}{\mu} v_2=0,\\
		-\Delta v_3-\frac{z}{\lambda +2\mu} v_3=0,
	\end{system}
\end{equation}
and a solution $u$ to~\eqref{eq:free-evl} is given by
\begin{equation*}
	u=\mathcal{F}^{-1}P \mathcal{F} v = i
	\begin{pmatrix}
		-\partial_2 v_1 -\partial_3 v_2 + \partial_1 v_3\\
		\partial_1 v_1 + \partial_2 v_3\\
		\partial_1v_2 + \partial_3 v_3
	\end{pmatrix}.
\end{equation*} 
For $k :=(0, z/\mu, 0)$, the function $v (x) := (\mu e^{i k \cdot x} / z, 0, 0)$ is clearly solution to \eqref{eq:Helmholtz-type}.
So, the function $u (x) = (e^{i k \cdot x},0,0)$ is solution to
\eqref{eq:free-evl}, $| u (x)| =1$ for almost all $x\in \R^3$ and all its
derivatives are bounded.
\end{proof}
With this result at hand, we are in position to prove the following theorem guaranteeing that the semi-axis $[0,\infty)$ belongs to the spectrum of $-\Delta^\ast +V.$
\begin{proposition}\label{thm:inclusion}
Let $d\geq 3$ and assume~\eqref{eq:cond-FKV-d3}. Then $[0,\infty) \subset \sigma(-\Delta^\ast + V).$
\begin{proof}
Let $z\in (0,\infty).$ We construct the sequence $\{\phi_n\}_{n\in  \N}$
from Lemma~\ref{lemma:FKV} applied to $H=-\Delta^\ast + V$ and $z$,
making use of Lemma~\ref{lemma:diagonalization}. Let $u$ be as in
Lemma~\ref{lemma:diagonalization}: we set
$\phi_n(x)=\varphi_n(x)u(x),$ where
$\varphi_n(x):=n^{-d/2}\varphi_1(x/n)$ for all $n\geq 1$
, and $\varphi_1\in
C^{\infty}_0(\R^d),$ $\|\varphi_1\|_{L^2(\R^d)}=1.$ Clearly
\begin{equation}\label{eq:limit-infty}
	\|\varphi_n\|_{L^2(\R^d)}=\|\varphi_1\|_{L^2(\R^d)}=1,
	\;
	\|\nabla \varphi_n\|_{L^2(\R^d)}=n^{-1}\|\nabla \varphi_1\|_{L^2(\R^d)},
	\;
	\|\partial_j \partial_k \varphi_n\|_{L^2(\R^d)}=n^{-2}\|\partial_j\partial_k\varphi_1\|_{L^2(\R^d)}, 
\end{equation}  
for any $j,k=1,2,\dots, d.$
Notice that as $u$ is chosen such that $|u(x)|=1,$ then $\|\phi_n\|_{L^2(\R^d)^d}=1$ and clearly $\phi_n\in \mathcal{D}(h)=\mathcal{D}(h_0)=H^1(\R^d)^d$ for all $n\in \N.$
Moreover, using that $u$ satisfies~\eqref{eq:free-evl} and that $u$
and its derivatives are bounded, one has
(we hide the summation over repeated symbols)
\begin{multline}
\label{eq:sing-seq}	
		\|-\Delta^\ast \phi_n-z\phi_n\|_{L^2(\R^d)^d}\\
		\begin{aligned}  
			&=\|-\mu \Delta\varphi_n u -2\mu (\nabla\varphi_n \cdot \nabla) u 
		-(\lambda + \mu)\nabla \varphi_n \div u 
		-(\lambda+\mu)\partial_j \nabla \varphi_n u_j
		-(\lambda + \mu)\partial_j \varphi_n \nabla u_j\|_{L^2(\R^d)^d}\\
		&\leq \mu \|\Delta \varphi_n\|_{L^2(\R^d)^d} \|u\|_{L^\infty(\R^d)^d}
		+2\mu \|\partial_j \varphi_n\|_{L^2(\R^d)^d}\|\partial_j u\|_{L^\infty(\R^d)^d}\\
		&\phantom{=}+(\lambda + \mu) \|\nabla \varphi_n\|_{L^2(\R^d)^d}\|\div u\|_{L^\infty(\R^d)^d}
		-(\lambda + \mu) \|\partial_j \nabla \varphi_n\|_{L^2(\R^d)^d}\|u_j\|_{L^\infty(\R^d)^d}\\
		&\phantom{=}+(\lambda + \mu)\|\partial_j \varphi_n\|_{L^2(\R^d)^d}\|\nabla u_j\|_{L^\infty(\R^d)^d}.
	\end{aligned}
\end{multline}
From~\eqref{eq:limit-infty} it follows that the right hand side of~\eqref{eq:sing-seq} goes to zero as $n$ tends to infinity.

Using the Hardy-type subordination~\eqref{eq:cond-FKV-d3} one has
\begin{equation}\label{eq:v-sing-seq}
		|v[\phi_n]| =\Big|\int_{\R^d} \overline{V\phi_n}\cdot \phi_n \Big|
		\leq \||V|^{1/2}\varphi_n\|_{L^2(\R^d)^d}^2\\
		\leq a \|\nabla \varphi_n\|_{L^2(\R^d)^d}^2,
\end{equation}
again using~\eqref{eq:limit-infty} it follows that the right hand side of~\eqref{eq:v-sing-seq} goes to zero as $n$ tends to infinity.
The numerator in~\eqref{eq:sequence-spectr} can be estimated as follows
\begin{equation*}
	\begin{split}
		|h(\phi_n,\psi) -z(\phi_n,\psi)|
		&=|(-\Delta^\ast \phi_n -z\phi_n, \psi) +v(\phi_n,\psi)|\\
		&\leq \|-\Delta^\ast \phi_n -z\phi_n\|_{L^2(\R^d)^d}\|\psi\|_{L^2(\R^d)^d}
		 +\sqrt{|v[\phi_n]|} \sqrt{|v[\psi]|}\\
		&\leq \|-\Delta^\ast \phi_n -z\phi_n\|_{L^2(\R^d)^d}\|\psi\|_{L^2(\R^d)^d}
		 +\sqrt{|v[\phi_n]|}\sqrt{a} \|\nabla \psi\|_{L^2(\R^d)^d}\\
		 &\leq 2\big(\|-\Delta^\ast \phi_n -z\phi_n\|_{L^2(\R^d)^d} + \sqrt{|v[\phi_n]|}\sqrt{a}\big)\|\psi\|_{\mathcal{D}(h_0)},
	\end{split}
\end{equation*}
where $\|\cdot\|_{\mathcal{D}(h_0)}$ is the usual $H^1(\R^d)^d$ norm.
As for the denominator in~\eqref{eq:sequence-spectr}, using again~\eqref{eq:cond-FKV-d3}, it follows
\begin{equation*}
	\begin{split}
		\|\psi\|_{\mathcal{D}(h)}^2
	&=\mu\|\nabla \psi_\textup{S}\|_{L^2(\R^d)^d}^2 + (\lambda + 2\mu)\|\nabla \psi_\textup{P}\|_{L^2(\R^d)^d}^2 + \Re v[\psi] + \|\psi\|_{L^2(\R^d)^d}^2\\
	&\geq (\min\{\mu, \lambda + 2\mu\}-a)\|\nabla \psi\|_{L^2(\R^d)^d}^2 + \|\psi\|_{L^2(\R^d)^d}^2\\
	&\geq \min\{1,(\min\{\mu, \lambda + 2\mu\}-a)\}\|\psi\|_{\mathcal{D}(h_0)}^2.
	\end{split}
\end{equation*}
Using the previous estimates one has
\begin{equation*}
	\sup_{\substack{\psi\in \mathcal{D}(h)\\ \psi\neq 0}}
	\frac{|h(\phi_n, \psi)-z(\phi_n,\psi)|}{\|\psi\|_{\mathcal{D}(h)}}
	\leq 2\frac{\|-\Delta^\ast \phi_n -z\phi_n\| + \sqrt{|v[\phi_n]|}\sqrt{a}}{\sqrt{\min\{1,(\min\{\mu, \lambda + 2\mu\}-a)\}}}.
\end{equation*}
Since the right hand side tends to zero due to~\eqref{eq:sing-seq}
and~\eqref{eq:v-sing-seq}, the sequence $\phi_n$ satisfies the
hypotheses of Lemma~\ref{lemma:FKV}, thus $(0,\infty)\subset
\sigma(-\Delta^\ast +V).$ Since the spectrum is closed, we get the thesis.
	\end{proof}
\end{proposition}
\subsection{Absence of residual spectrum}
In order to conclude the claimed stability, it is left to show that the residual spectrum of $-\Delta^\ast + V$ is empty. This is the object of the next theorem.
\begin{proposition}\label{lemma:residual}
	Let $d\geq3.$ Then $\sigma_\textup{r}(-\Delta^\ast + V)=\varnothing.$
	\begin{proof}
		Let define $H_V:=-\Delta^\ast + V.$ It is easy to see that $H_V^\ast=H_{\overline{V}^{\,t}},$ where $\overline V^{\,t}$ denotes the conjugate transpose of the matrix $V.$ 
		
Let denote with $J$ the complex-conjugation
transposition operator defined by
$J(Au)=\overline{A}^{\,t} \overline{u},$ for any
square matrix $A\in \C^{d\times d}$ and any vector
$u\in \C^d$: notice that $J(u)=J(I_{\C^d}
u)=\overline{u},$ with $I_{\C^d}$ the $d\times d$
identity matrix, in other words, given any vector
$u=I_{\C^d} u$ then $J$ acts as the usual
complex-conjugation operator. $J$ as defined above is
a conjugation operator in the sense
of~\cite[Sec. III.5]{Ed_Ev}. One easily checks that $H_V^\ast=JH_VJ,$
\emph{i.e.}, $H_V$ is $J$-self-adjoint, and thus it has no residual spectrum (\emph{cfr.}~\cite{BK08}) .
	\end{proof}
\end{proposition}

\subsection{Proofs of \Cref{thm:dim3}, \Cref{thm:M-Cd3} and \Cref{thm:L^p-d3}}
\label{sub:dim=3}

\begin{proof}[Proof of \Cref{thm:dim3}]
The proof of \Cref{thm:dim3} follows from
\Cref{thm:absence-evs-d3},
\Cref{thm:absence-continuous}, \Cref{thm:inclusion} and
\Cref{lemma:residual}.
\end{proof}

Now we turn to the proof of \Cref{thm:M-Cd3}. We stress that the validity of Propositions~\ref{thm:absence-evs-d3}--\ref{lemma:residual} (from which the stability of the spectrum of $-\Delta^\ast + V$ follows) requires only two ingredients: first, one needs, $\mathcal{D}(v)\subset \mathcal{D}(h_0)$ (\emph{cfr}.~\eqref{eq:h_0} and~\eqref{eq:v}) and secondly $\|K_z\|\leq\mathfrak{a}<1.$ As soon as we consider class of potentials such that these two requests are satisfied, then one gets spectral stability of the perturbed Lamé operators with such perturbations as a consequence of Propositions~\ref{thm:absence-evs-d3}--\ref{lemma:residual}. This allows us to prove \Cref{thm:M-Cd3} and \Cref{thm:L^p-d3}.

\begin{proof}[Proof of \Cref{thm:M-Cd3}]
Thanks to the Hardy-type inequality~\eqref{eq:Hardy} with~\eqref{eq:d=3-a}, if $V\in
\mathcal{L}^{2,p}(\R^3),$ $1<p\leq 3/2$ then $\mathcal{D}(v)\subset
\mathcal{D}(h_0).$ Moreover, from~\eqref{BS-MC-d=3} and hypothesis~\eqref{eq:cond-d3-M-C} one has $\|K_z\|\leq\mathfrak{a}<1.$  In light of the remark above, this concludes the proof.
\end{proof} 

\begin{proof}[Proof of \Cref{thm:L^p-d3}]
Thanks to~\eqref{eq:Sob-Hol}, if $V\in L^{3/2}(\R^3)$ then $\mathcal{D}(v)\subset \mathcal{D}(h_0).$ Moreover, from~\eqref{BS-Lp-d=3} and assumption~\eqref{eq:cond-d3-Lp} one has 
$\|K_z\|\leq \mathfrak{a}<1.$ This concludes the proof. 
\end{proof}

\providecommand{\bysame}{\leavevmode\hbox to3em{\hrulefill}\thinspace}


\end{document}